\documentclass[reqno,11pt,final]{amsart}
\usepackage{style}
\title[Kolmogorov equations for filtering processes]{Kolmogorov equations on spaces of measures associated to nonlinear filtering processes}

\author{Mattia Martini}
\address{Mattia Martini: Dipartimento di Matematica "Federigo Enriques", Università degli Studi di Milano, Via Saldini 50,  20133 Milano (Italy)}
\email{mattia.martini@unimi.it}

\begin{document}
\begin{abstract}
We introduce and study some backward Kolmogorov equations associated to filtering problems. In the stochastic filtering framework, SDEs for measure-valued processes arise naturally (Zakai and Kushner-Stratonovich equation). The associated Kolmogorov equations have been intensively studies, assuming that the measure-valued processes admit a density and then by exploiting stochastic calculus in Hilbert spaces. 

Our approach differs from this since we do not assume the existence of a density and we work directly in the context of measures. We first formulate two Kolmogorov equations on spaces of measures, and then we prove existence and uniqueness of classical solutions.
\end{abstract}
\maketitle

\section{Introduction}

The main aim of this paper is to study the backward Kolmogorov equations of parabolic type associated to measure-valued processes arising in the  context of stochastic filtering. The principal result is about existence and uniqueness of classical solutions to these partial differential equations, although other intermediate results are of independent interest.\\

The study of measure-valued stochastic processes is a classical topic that has attracted an enormous interest. For instance, there is a large  literature related to the superprocesses framework (see e.g. \cite{dawsonsflour2}), but more recently it has been related to the topic of mean field games and McKean-Vlasov equations (see \cite{cardaliaguet1, lions} or \cite{carmonadelarue1,carmonadelarue2}), where probability measure-valued processes are used in problems with common noise to describe the evolution of the conditional laws of some finite dimensional stochastic processes. Thanks to this recent interest, many new results are now available, such as It\^o formulas (\cite{carmonadelarue1,carmonadelarue2}) and tools for differential calculus on spaces of measures (\cite{cardaliaguet1, carmonadelarue1, lions}). Moreover, a topic of great interest are the partial differential equations on spaces of probability measures associated to these problems, such as, for instance, the so-called master equation in the context of mean field games (see for instance \cite{cardaliaguetdelaruelasrylions, carmonadelarue2}), the backward Kolmogorov equation associated to McKean-Vlasov type equation (see for instance \cite{buckdahnlipengrainer}), or certain Hamilton-Jacobi equations (\cite{gangbotudorascu}).\\

Our work gives a contribution in this direction. Differently from the previous contexts, our aim is to study partial differential equations on space of measures associated to measure-valued processes arising in stochastic filtering problems. In particular, given a measure-valued process, we first introduce the so called backward Kolmogorov equation associated to it, which is a partial differential equation on a space of measures. Then, we study existence and uniqueness of its classical solutions.

Stochastic filtering has been intensively studied, see for instance \cite{baincrisan, xiong} and the references therein for a systematic exposition of the topic. Two basic notions of the theory are the so-called normalized and unnormalized filtering processes, which are a probability measure-valued process and a positive measure-valued process respectively,  and are proved to be the solutions, in a sense that will be clarified later, to  stochastic differential equations, called the Kushner-Stratonovich and the Zakai equation respectively.

A classical way to deal with these equations (see, for instance, \cite{pardoux, rozovsky} ) is to show that the solutions admit a density with respect to the Lebesgue measure, which possibly belongs to a suitable Hilbert space of functions. Thus, one can study the density processes instead of the measure-valued processes and rely on tools of stochastic calculus on Hilbert spaces to further explore their properties. The price to pay is the  introduction of 
unnecessary assumptions entailing that the filtering processes have a density. In this paper we avoid these conditions and rather follow the approach of \cite{baincrisan,bhattkallianpurkarandikar, kurtzocone,heunislucic,szpirglas}, where the filtering processes are studied as genuine
measure-valued processes. 

 In this framework, the Zakai equation reads as 
\begin{equation*}
	\de\scalprod{\rho_t}{\psi} = \scalprod{\rho_t}{A\psi}\de t + \scalprod{\rho_t}{h\psi + B\psi}\cdot\de Y_t,
\end{equation*} 
where the solution $\rho = \{\rho_t,t\in[0,T]\}$ is a positive measure-valued process, $A,B$ are differential operators defined by the formulae
\begin{align*}
	A\psi(x) &\defeq \sum_{i=1}^d f(x)\partial_i \psi(x) + \frac{1}{2}\sum_{i,j=1}^d(\sigma\sigma^\top)_{ij}(x)\partial_{ij}\psi(x)+(\bar\sigma\bar\sigma^\top)_{ij}(x)\partial_{ij}\psi(x),\quad x\in\R^d,\\
	B_k\psi(x) &\defeq \sum_{i=1}^d \bar{\sigma}_{ik}(x)\partial_i\psi(x) ,\quad x\in\R^d,k=1,\dots d,
\end{align*} 
and $f,\sigma,\bar\sigma,h$ are functions that have to satisfy some hypotheses we will formulate later. The Kushner-Statonovich equation reads as
\begin{equation*}
	\de\scalprod{\Pi_t}{\psi} = \scalprod{\Pi_t}{A\psi}\de t + \left (\scalprod{\Pi_t}{h\psi + B\psi} - \scalprod{\Pi_t}{\psi}\scalprod{\Pi_t}{h}\right)\cdot\de I_t,
\end{equation*}
where $\Pi = \{\Pi_t,[0,T]\}$ is probability measure-valued. In the previous equations the processes $Y$ and $I$ are Brownian motions (with respect to appropriate probability measures) and we use the notation $\scalprod{\mu}{\psi} = \int \psi(x)\mu(\de x)$. The equalities are understood to hold for every $\psi$ in a certain class of test functions. In the following we also denote with $\mathcal{M}^+(\R^d)$ and $\mathcal{P}(\R^d)$ the spaces of positive and probability measures on $\R^d$ respectively.\\

Our main results are two theorems on existence and uniqueness of classical solutions to the backward Kolmogorov equations, associated to the Zakai and the Kushner-Stratonovich equations, introduced here for the first time.
The solutions are functions $u: [0,T]\times \mathcal{M}^+(\R^d)\to \R$
or $u: [0,T]\times \mathcal{P}(\R^d)\to\R$
respectively.
 We note that 
finding  solutions to Kolmogorov equations on infinite dimensional spaces is a challenging problem and it has been studied intensively, see for instance   \cite{dapratozabczyk3},  and
the search of classical solutions is often addressed, as in 
 \cite{flandolizanco}.  Most  results are only concerned with  the Hilbert space case, namely when 
$u: [0,T]\times H\to \R$, where $H$ is a Hilbert space. The extension to spaces of measure requires entirely different methods and in particular new tools from differential calculus, as we will 
explain later.

The first result, given in Theorem \ref{thm: exbkwzakai}, concerns the backward Kolmogorov equation associated to the Zakai equation, that reads as
\begin{equation}\label{eqn: kolzakintro}
	\begin{cases}
		\partial_su(\mu,s) + \mathcal{L}u(\mu,s) = 0,\quad&(\mu,s)\in\mathcal{M}^+(\R^d)\times[0,T],\\
		u(\mu,T) = \Phi(\mu),\quad&\mu\in\mathcal{M}^+(\R^d),
	\end{cases}
\end{equation}
where 
\begin{equation*}
	\begin{aligned}
		\mathcal{L} u(\mu) &= \mu\left(\diff _\mu u(\mu)\cdot f\right)  
		+\frac{1}{2}\mu\left(\tr\left\{\diff _x\diff _\mu u(\mu)\sigma\sigma^\top\right\}\right)\\ &+ \frac{1}{2}\mu\left(\tr\left\{\diff _x\diff _\mu u(\mu)\bar{\sigma}\bar{\sigma}^\top\right\}\right)
 + \frac{1}{2}\mu\otimes\mu\left(\lf^2 u(\mu)h\cdot h\right)\\ &+ \mu\otimes\mu\left(h\cdot\bar{\sigma}^\top\lf \diff _\mu u(\mu)\right) + \frac{1}{2}\mu\otimes\mu\left(\tr\left\{\diff _\mu^2 u(\mu)\bar{\sigma}\bar{\sigma}^\top\right\}\right),
	\end{aligned} 
\end{equation*}
and $\lf u, \lf^2 u, \diff _\mu u,  \diff ^2_\mu u $ are notions of first and second-order derivatives on $\mathcal{M}^+(\R^d)$ we will discuss later, whilst $\diff _x$ denotes the gradient on $\R^d$. In Theorem \ref{thm: exbkwzakai} we show that if the terminal condition $\Phi$ is regular enough  then there exists a unique classical solution to \eqref{eqn: kolzakintro} (defined for $\mu$ in an appropriate subset of $\Mp$). Analogously, in Theorem \ref{thm: exuniqkskolm} we prove existence and uniqueness for classical solutions to the  backward Kolmogorov equation on $\mathcal{P}(\R^d)$ associated to the Kushner-Stratonovich equation, that is
\begin{equation}\label{eqn: kolksintro}
\begin{cases}
		\partial_su(\pi,s) + \mathcal{L}^{KS}u(\pi,s) = 0,\quad&(\pi,s)\in\prob(\R^d)\times[0,T],\\
		u(\pi,T) = \Phi(\pi),\quad&\pi\in\prob(\R^d),
	\end{cases}
\end{equation}
where
\begin{equation*}
	\begin{aligned}
		\mathcal{L}^{KS} u(\pi) &= \pi\left(\diff _\pi u(\pi)\cdot f\right) \\ 
		&+\frac{1}{2}\pi\left(\tr\left\{\diff _x\diff _\pi u(\pi)\sigma\sigma^\top\right\}\right)+ \frac{1}{2}\pi\left(\tr\left\{\diff _x\diff _\pi u(\pi)\bar{\sigma}\bar{\sigma}^\top\right\}\right)\\
& + \frac{1}{2}\pi\otimes\pi\left(\delta_\pi^2 u(\pi)h\cdot h\right)+ \pi\otimes\pi\left(h\cdot\bar{\sigma}^\top\delta_\pi \diff _\pi u(\pi)\right)\\
& + \frac{1}{2}\pi\otimes\pi\left(\tr\left\{\diff _\pi^2 u(\pi)\bar{\sigma}\bar{\sigma}^\top\right\}\right) + \frac{1}{2}[\pi(h)\cdot\pi(\pi)]\pi\otimes\pi\left(\delta_\pi^2 u (\pi)\right)\\
& - \pi\otimes\pi\left(\delta_\pi^2 u (\pi)h\right)\cdot\pi(h) - \pi\otimes\pi\left(\bar{\sigma}^\top\delta_\pi\diff _\pi u(\pi)\right)\cdot\pi(h).
	\end{aligned}
\end{equation*}
In both cases, the solution is given by a probabilistic representation formula related to the filtering processes solutions to the Zakai and Kushner-Stratonovich equations. 
\\

In the study of \eqref{eqn: kolzakintro} we have to face a specific difficulty: since
we deal with functions over $\Mp$  we can not rely on the various notions in differential calculus that have been developed in the last years with reference to the space of probability measures (see \cite{ambrosiogiglisavare,carmonadelarue1,gangbotudorascu,lions} for different notions of derivative and a comparison among them). For instance, the technique introduced by P.-L. Lions in \cite{lions} where the problem is lifted on a space of random variables is no longer available. However, it turns out that the notion of linear functional derivative, given in \cite{carmonadelarue1}, can be extended to our framework. More precisely, we say $\lf u\colon\Mp\times\R^d\to\R$ is the derivative of $u\colon\Mp\to\R$ in linear functional sense if $u$ and $\lf u$ have some regularity properties and if for every $\mu,\mu'\in\Mp$ it holds
\begin{equation*}
	u(\mu')-u(\mu) = \int_0^1\int_{\R^d}\lf u\left(t\mu' + (1-t)\mu,x\right) [\mu'-\mu](\de x) \de t.
\end{equation*}
Once $\lf u$ has been introduced, one can set $\diff _\mu u (\mu,x) \defeq \diff _x\lf u (\mu,x)$, $(\mu,x)\in\Mp\times\R^d$. If we restrict ourselves to the space of probability measures with finite second moment, under certain hypotheses this last definition coincides with the notion of derivative introduced by Lions, which we will call $L$-derivative. Most of the It\^o formulas available in the literature involve only $L$-derivatives, whilst in our case both linear functional and $L$-derivatives are needed. We point out that in the literature other notions of derivative for functionals over sets of positive measures have been introduced. For instance, a derivative over $\Mp$ has been introduced in the framework of measure-valued processes related to particle systems (see for instance \cite{dawsonsflour2} and \cite{kolokoltsov}) and it has been intensively used in the context of Fleming-Viot processes. It turns out that, under certain conditions, this notion coincides with the one adopted in this paper (see \cite{renwang}). Another example is \cite{albeveriokondratievrockner}, where the authors give a definition of derivative for functionals over Poisson spaces.
\\

As expected, in order to show the uniqueness property in \eqref{eqn: kolzakintro} and \eqref{eqn: kolksintro} one needs to prove a suitable It\^o formula, in our case for the composition of a real-valued function and a measure-valued process. 
In the recent literature formulas of this kind have been proved when the process takes values in a space of probability measures, see for instance \cite{buckdahnlipengrainer, cardaliaguetdelaruelasrylions, carmonadelarue2} and it is constructed as the time evolution of the one-point marginal law (in certain situations  the one-point conditional marginal law) of a given finite dimensional process.
For our purposes we need very different results.
In Proposition \ref{prop: itoz}  we provide an It\^o formula for the composition of
 a real-valued function over $\mathcal{M}^+(\R^d)$ and the
 $\mathcal{M}^+(\R^d)$-valued 
 process solution to the Zakai equation.
Similarly, in Proposition \ref{prop: itoks} we prove an It\^o formula related to the $\prob(\R^d)$-valued process solution to the Kushner-Stratonovich equation.  
Both results are new, but the latter 
 can be viewed as a generalization of the one obtained in the context of mean field games with common noise (\cite[Section 4.3]{carmonadelarue2}), as explained in greater detail in Remark \ref{rmk: whymyitoiscool}.
 One major technical difference from existing cases is the fact that
the  Zakai and  Kushner-Stratonovich  equations are understood to hold
in a weak form, namely for arbitrary choice of the occurring test function.
 In our proof we first show the formula for a smaller class of functions with good properties by exploiting the classical It\^o formula, then we obtain the general result by an approximation argument. In order to pass to the limit in the It\^o formula one needs convergence 
of the approximating functions as well as their first and second derivatives (linear functional and $L$-derivatives). The required constructions have some interest in themselves and can be used again in similar contexts.
\\
  
Concerning the existence of classical solutions to \eqref{eqn: kolzakintro} and \eqref{eqn: kolksintro}, the most difficult part of the proof is the investigation of the regularity of the solutions to the Zakai and the Kushner-Stratonovich equations with respect to the initial datum. 
Dependence of the filtering processes on the initial condition has been the object of intense study, since it is related to the problem of assessing the effect of a misspecification of the initial distribution of the signal process in the filtering problem. However, the study of the differentiability properties of the solution with respect to the intial conditions seems to be addressed here for the first time. As this relates to differentiability of measure-valued processes with respect to a measure (the starting point of the process itself) we need to introduce novel notions of differentiability for mappings from $\Mp$ (or $\prob(\R^d))$ to $\Mp$(or $\prob(\R^d))$.\\

This work lays foundations for the study of partial differential equations on spaces of measures associated to stochastic filtering problems. Object of further research will be the existence and uniqueness in viscosity sense of solutions to the Kolmogorov equations introduced in this work, under less restrictive conditions. Later, non-linear partial differential equations will be considered, in particular the Hamilton-Jacobi-Bellman equation arising from the optimal control problem with partial observation problem will be investigated, see for instance the book \cite{bensoussan} for a systematic introduction to the problem or the recent paper \cite{bensoussanyam} for a modern approach in the density case, based on mean field techniques. This problem has been already tackled for the Zakai equation in \cite{bandinicossofuhrmanpham1,bandinicossofuhrmanpham2} exploiting the randomization method and BSDEs techniques, but in the more restrictive case where the function $h$ is identically equal to zero.
A look at \eqref{eqn: kolzakintro} shows that this assumption allows the authors to rely only on $L$-derivatives and exploit previous results on well-posedness of related partial differential equations equations,
an approach which is not possible in our situation.\\

To conclude, we describe the plan of the paper. In Section \ref{sec: 2} we introduce and discuss the notions of derivatives needed later. In Section \ref{sec: approx} we provide the approximation results for real valued functions over $\Mp$, which play a key role in the proofs of the It\^o formulas. In Section \ref{sec: eqnfilter} we briefly recall the filtering problem,  we introduce the Zakai and the Kushner-Stratonovich equations and we state the hypotheses we will adopt for the rest of the paper. In Section \ref{sec: zakai} we state and prove  results on the solution to the Zakai equation, such as the It\^o formula and the regularity of the solution with respect to the initial datum. Finally, in Section \ref{sec: zakaikolmogorov} we state and prove the existence and uniqueness theorem for the classical solutions to the backward Kolmogorov equation associated to the Zakai equation, whilst in Section \ref{sec: kseqn} we do the same for the Kushner-Stratonovich equation.

\subsection{Notation and preliminaries}
We collect here some recurrent notations we will use during all our discussion. Regarding the space of continuous functions, we denote with $\mathrm{C}^k(\R^d)$, $k\in\N$, the space of real-valued functions over $\R^d$ which are $k$-times continuously differentiable ($k=0$ is omitted and denotes the space of continuous functions) and with $\mathrm{C}^k_\mathrm{b}(\R^d)$ the functions in $\mathrm{C}^k(\R^d)$ bounded and with bounded derivatives up to order $k$. We endow the space of continuous functions with the infinity norm, namely $\norm {u}_\infty \defeq \sup_{x\in\R^d} |u (x)|$ if $u\in \mathrm{C}(\R^d)$. If $u$ depends on several argument, $\norm{u}_\infty$ denotes the infinity norm where the supremum is taken over all the arguments. Analogously, if $u\in\mathrm{C}^k(\R^d)$ we call $\mathrm{C}^k$ norm the quantity given by $\norm{u}_{\mathrm{C}^k}\defeq \norm{u}_\infty + \sum_{i=1}^k \norm{\diff _x^i u}_\infty$
We also denote with $\mathrm{B_b}(\R^d)$ the space of Borel measurable and bounded functions over $\R^d$.\\

Let $K$ be a Borel subset of $\R^d$. We denote with $\mathcal{M}(K)$ the set of signed measures over $K$ with finite total variation, with $\mathcal{M}^+(K)$ the set of positive finite measures over $K$ and with $\prob(K)$ the set of probability measures over $K$, that is the subset of $\mathcal{M}^+(K)$ made by the measures with unitary total mass. For $p\in[1,+\infty)$, we denote the spaces of measures with finite $p$-th moment by $\mathcal{M}_p(K)$, $\mathcal{M}^+_p(K)$,$\mathcal{P}_p(K)$. More precisely,
\begin{equation*}
	\mathcal{M}_p(K)\defeq \{\mu\in\mathcal{M}(K)\colon\int_K|x|^p\mu(\de x)<\infty\},
\end{equation*}
and the same for $\mathcal{M}^+_p(K)$ and $\mathcal{P}_p(K)$. Regarding the integration of functions, we will denote by $\mu(\psi)$ or by $\scalprod{\mu}{\psi}$ the quantity $\int\psi(x)\mu(\de x)$, for a measure $\mu$ and an integrable function $\psi$. We denote the total mass given to $K$ by $\mu$ with $\mu(K)$ or $\scalprod{\mu}{\mathbf{1}}$, where $\mathbf{1}(x) = 1$ for every $x\in K$. We say that a sequence $\{\mu_n\}_{n\geq 1}$ converges weakly to a measure $\mu$ if $\mu_n(\psi)\to\mu(\psi)$ for any $\psi\in\mathrm{C_b}(K)$. We can notice that in order to have weak convergence, it is enough to check that $\mu_n(\psi)\to\mu(\psi)$ for any $\psi\in\mathrm{C^2_b}(K)$.\\

For the results in Section \ref{sec: approx}, it will be useful to introduce a distance over $\prob_p(K)$. We define the Wasserstein distance of order $p\geq 1$ between $\pi,\pi'\in\prob_p(K)$ as 
\begin{equation*}
	W_p(\pi,\pi') = \inf\left\{\left(\int_{K\times K}\!\! |x-x'|^p\gamma(\de x,\de x')\right)^{\frac{1}{p}}\!\!\!\colon \!\gamma\in\prob_p(K\times K) \text{ with marginals $\pi,\pi'$}\right\}.
\end{equation*}
To conclude, we point out that $(\prob_p(K), W_p)$ is a complete and separable metric space and that the convergence in $W_p$ implies the weak convergence stated before. A very detailed discussion on these topics can be found for instance in \cite{ambrosiogiglisavare} or from a more probabilistic point of view in \cite{carmonadelarue1}.

\section{Differential calculus on spaces of positive measures}\label{sec: 2}
Since our final goal is to discuss some Kolmogorov equations on suitable spaces of measures, we need to introduce notions of derivatives for real-valued or measure-valued functions over spaces of measures. Regarding the real-valued functions, we take inspirations from the literature recently developed for real-valued function over the space of probability measures. We give a little extension of the notion of linear functional derivative (or flat derivative) discussed for instance in \cite{cardaliaguetcirantporretta,cardaliaguetdelaruelasrylions,carmonadelarue1}. 
Another definition we need mimics the derivative introduced in the context of mean field games and discussed for instance in \cite{cardaliaguet1, carmonadelarue1}. Following \cite{cardaliaguetdelaruelasrylions}, we define it as the spatial gradient of the linear functional derivative. For probability measures, under proper hypotheses, this definition coincides with the original one given by Lions in \cite{lions} through the lifting on a Hilbert space. A discussion on the relations among these definition can be found in \cite{carmonadelarue1} in the case of probability measures or in \cite{renwang} in a more general case. The last definition we introduce is a notion of derivative of functions from $\Mp$ to $\Mp$. This is a new definition, strongly inspired by the previous ones. 
\begin{definition}[Linear functional derivative]\label{def: lfder}
	A function $u\colon\Mp\to\R$ is said to have linear functional derivative if it is continuous, bounded and if there exists a function
	\begin{equation*}
		\lf u\colon\Mp\times\R^d\ni(\mu,x)\mapsto \lf u(\mu,x)\in \R,
	\end{equation*}
	that is bounded, continuous for the product topology, $\Mp$ equipped with the weak topology, and such that for all $\mu$ and $\mu'$ in $\Mp$, it holds:
	\begin{equation}\label{eqn: fder}
		u(\mu')-u(\mu) = \int_0^1\int_{\R^d}\lf u\left(t\mu' + (1-t)\mu,x\right) [\mu'-\mu](\de x) \de t.
	\end{equation}
	We call $\mathrm {C^1}(\Mp)$ the class of functions from $\Mp$ to $\R$ that are differentiable in linear functional sense.
\end{definition}
\begin{remark}
	If $u\in \mathrm {C^1}(\Mp)$, we can introduce a notion of second-order derivative by asking that the mapping $\mu\mapsto\lf u(\mu,x)$ is differentiable in linear functional sense for every $x$ and that $(\mu,x,y)\mapsto \lf^2 u(\mu,x,y)$ is bounded and continuous. In general one can define derivatives of order $k\in\N$ and introduce the space $\mathrm {C}^k(\Mp)$ of functions that are $k$ times differentiable in linear functional sense. Notice that every time we differentiate, the derivative depends on a new spatial variable.
\end{remark}
We introduce now the second notion of derivative, namely the $L$-derivative, for real-valued functions over $\M$. We follow the definition given in \cite{cardaliaguetdelaruelasrylions}, since for positive measure we cannot rely on the lifting procedure of \cite{lions}. 
\begin{definition}[$L$-derivative]\label{def: ldiff}
	A function $u$ in $\mathrm {C^1}(\Mp)$ is said to be $L$-differentiable if, for every $\mu\in\Mp$, the mapping $\R^d\ni x\mapsto\lf u(\mu,x)\in\R$ is everywhere differentiable, with $\Mp\times\R^d\ni(\mu,x)\mapsto \diff _x\lf u (\mu,x)\in \R^d$ continuous and bounded. 
	We set:
	\begin{equation}
		\diff _\mu u(\mu,x) \defeq \diff _x\lf u(\mu, x)\in\R^d,
	\end{equation}
	and we denote this class of functions with $\mathrm{C^1_L}(\Mp)$. 
\end{definition}
\begin{remark}
	If we consider functions over $\ptwo$ (see Remark \ref{rmk: derivativeonp2}), Definition \ref{def: ldiff} turns out to coincide with the definition of $L$-derivative introduced by Lions in \cite{lions} and discussed for instance in \cite{cardaliaguet1,carmonadelarue1}. More relations with other notions of derivative in the case of signed measures have been also investigated in \cite{renwang}. 
\end{remark}
Regarding the second-order $L$-derivative, again in view of \cite{cardaliaguetdelaruelasrylions}, we give the following definition: 
\begin{definition} \label{def: c2lderivatives}
	A function $u\colon\Mp\to\R$ is said to be in $\mathrm {C^2_L}(\Mp)$ if the following conditions hold:
	\begin{itemize}
		\item[a.] $u$ is in $\mathrm {C^2}(\Mp)$; 
		\item[b.] the mapping $\R^d\ni x\mapsto\lf u(\mu,x)\in\R$ is everywhere twice differentiable, with continuous and  bounded derivatives on $\Mp\times\R^d$; 
		\item[c.] the mapping $\R^d\times\R^d\ni(x,y)\mapsto\lf^2 u (\mu,x,y)$ is twice differentiable, with continuous and  bounded derivatives on $\Mp\times\R^d\times\R^d$.
	\end{itemize}
	We define the second-order $L$-derivative of $u\in \mathrm {C^2_L}(\Mp)$ as follows:
	\begin{equation*}
		\diff _\mu^2 u(\mu,x,y) \defeq \diff _x\diff _y^\top\lf^2 u (\mu,x,y)\in\R^{d\times d},
	\end{equation*}
	where $\mathrm D_y^\top = [\partial_{y_1},\dots ,\partial _{y_d}]$ is the gradient (with respect to $y$) operator seen as a row.
\end{definition}
\begin{remark}
	In order to define the second-order $L$-derivative $\diff_\mu^2 u$, it is enough to ask for less regularity of the mappings $x\mapsto\lf u(\mu,x)$ and  $(x,y)\mapsto\lf^2 u (\mu,x,y)$ (for instance $x\mapsto\lf u(\mu,x)$ can be once continuously differentiable). However, for our scopes, it is necessary to require further regularity of these mappings and so we included it in the definition to keep the exposition clearer.	
\end{remark}
\begin{remark}\label{rmk: derivativeonp2}
	If we consider $u\colon\prob(\R^d)\to\R$, then we ask that the condition \eqref{eqn: fder} in Definition \ref{def: lfder} holds for every $\mu,\mu'\in\prob(\R^d)$. In this case we will denotes with $ \mathrm {C^1}(\prob(\R^d))$ the space of all the functions differentiable in this sense. Of course, we can proceed in the same way for the derivatives of higher-order or for the $L$-derivatives. Moreover, this also works for subsets like $\mathcal{M}^+_p(\R^d)$ and $\mathcal{P}_p(\R^d)$, $p\in[1,+\infty)$.
\end{remark}

\begin{remark}
	If we consider functions defined over the space of probability measures $\prob(\R^d)$, we have that the linear functional derivative is defined up to an additive constant (see for instance \cite[Remark 5.46]{carmonadelarue1}). A way to guarantee uniqueness (see for instance \cite[Section 2.2.1]{cardaliaguetdelaruelasrylions}) is to adopt the convention
	\begin{equation}\label{eqn: convention}
		\int_{\R^d}\lf u (\mu, x)\mu (\de x) = 0,\quad \mu\in\prob(\R^d).
	\end{equation}
\end{remark}

\begin{remark}
	We can also give the definitions of linear functional and $L$-derivative in the case of measures with compact support $\mathcal{M}^+(K)$, $K\subset\R^d$ compact with sufficiently regular boundary. In this case the additional variable generated by the differentiation belongs to $K$ and the spatial differentiability required for the $L$-derivative has to be meant only in the proper direction at the boundary of $K$.
\end{remark}
\begin{remark}
	Let $(\mathrm {B},\norm{\cdot}_{\mathrm B})$ be a Banach space and let us consider $u\colon\Mp\to \mathrm B$. Then, all the previous definitions can be trivially extended to this framework. We will use the notations $\mathrm {C^k}(\Mp; \mathrm B)$, $\mathrm {C^1_L}(\Mp; \mathrm B)$ and $\mathrm {C^2_L}(\Mp; \mathrm B)$. In this case $\lf u\colon \Mp\times\R^d\to \mathrm B$ and $\diff _\mu u\colon\Mp\times\R^d\to \mathrm{B}^d$, and the same holds for higher-order derivatives.
\end{remark}
We conclude this first part of the section by presenting an example of computations of linear functional and $L$-derivatives.
\begin{example}\label{ex: cyl}
	Let $g\colon\R^n\to\R$ be in $\mathrm {C_b^2}(\R^n)$ and let $\{\psi\}_{i=1}^n\subset \mathrm {C_b}(\R^d)$. We define 
	\begin{equation*}
		u\colon\Mp\ni\mu\mapsto g\left(\scalprod{\mu}{\psi_1},\dots,\scalprod{\mu}{\psi_n}\right)\in\R.
	\end{equation*}
	Then $u\in \mathrm {C^2}(\Mp)$ and it holds:
	\begin{align*}
		\lf u(\mu,x) &= \sum_{k=1}^n\partial_kg\left(\scalprod{\mu}{\psi_1},\dots,\scalprod{\mu}{\psi_n}\right)\psi_k(x),\\
		\lf^2 u(\mu,x,y) &= \sum_{k,l=1}^n\partial_l\partial_kg\left(\scalprod{\mu}{\psi_1},\dots,\scalprod{\mu}{\psi_n}\right)\psi_k(x)\psi_l(y).
	\end{align*}
	Moreover, if $\{\psi_i\}_{i=1}^n\subset \mathrm {C_b^2}(\R^d)$, then $u$ is in $\mathrm {C_L^2}(\Mp)$ and it holds:
	\begin{align*}
		\diff _\mu u (\mu,x) & =  \sum_{k=1}^n\partial_kg\left(\scalprod{\mu}{\psi_1},\dots,\scalprod{\mu}{\psi_n}\right)\diff _x\psi_k(x),\\
		\diff _\mu^2 u (\mu,x,y) & = \sum_{k,l=1}^n\partial_l\partial_kg\left(\scalprod{\mu}{\psi_1},\dots,\scalprod{\mu}{\psi_n}\right)\diff _x\psi_k(x)\diff _y^\top\psi_l(y).
	\end{align*}
	These functions, which we call cylindrical, play a key role in the proof of the It\^o formula in Section \ref{sec: zakai}. We will discuss in Section \ref{sec: approx} some approximation properties of this class.
\end{example}

The last definition we give concerns the differentiability for functions from $\Mp$ to $\Mp$. The idea is to ask for a relation similar to \eqref{eqn: fder} for the measure-valued function tested against regular functions.
\begin{definition}[Linear functional derivative for measure-valued functions]\label{def: lfdermeas}
	We say that a function $m\colon\Mp\to\Mp$ is differentiable in linear functional sense if there exists a mapping
	\begin{equation*}
		\lfm m\colon\Mp\times\R^d\ni(\mu,x)\mapsto \lfm m(\mu,x)\in \Mp,
	\end{equation*}
	bounded in total variation, continuous for the product topology, $\Mp$ equipped with the weak topology, and such that for all $\mu$ and $\mu'$ in $\Mp$, it holds:			
	\begin{equation}\label{eqn: fderm}
		\scalprod{m(\mu')-m(\mu)}{\psi} = \int_0^1\int_{\R^d}\scalprod{\lfm m\left(t\mu' + (1-t)\mu, x\right)}{\psi} [\mu'-\mu](\de x) \de t,
	\end{equation}
	for every $\psi\in \mathrm {C_b}(\R^d)$.
	We call $\tilde{\mathrm C}^1(\Mp)$ the class of functions from $\Mp$ to $\Mp$ that are differentiable in linear functional sense and, analogously, we denote by $\tilde{\mathrm C}^k(\Mp)$ the space of functions that are $k$ times differentiable.
\end{definition}
\begin{remark}
	The joint continuity required in Definition \ref{def: lfdermeas} implies that, for every fixed $\mu\in\Mp$, the mapping $x\mapsto\lfm m (\mu,x)$ is measurable and so $\lfm u$ is a transition kernel.
\end{remark}

\begin{remark}\label{rmk: lflfm}
	It easy to see that if $m\colon\Mp\to\Mp$ is in $\mathrm{\tilde C^1}(\Mp)$ then the mapping $\mu\mapsto\scalprod{m(\mu)}{\psi}$ is in $\mathrm{C^1}(\Mp)$ for every $\psi\in\mathrm{C_b}(\R^d)$. In particular, $\scalprod{\lfm m (\mu, x)}{\psi} = \lf (\scalprod{m(\cdot)}{\psi})(\mu,x)$ for every $\psi\in\mathrm{C_b}(\R^d)$, $\mu\in\Mp$ and $x\in\R^d$. The converse is also true if we assume that the regularity of the mapping $\mu\mapsto\scalprod{m(\mu)}{\psi}$ is uniform with respect to $\psi$.
\end{remark}

\begin{example}
	Let us consider $\rho\in\mathrm{C_b}(\R^d;[0,+\infty))$, and let us define $m_\rho\colon\Mp\to\Mp$ as
	\begin{equation*}
		\scalprod{m_\rho(\mu)}{\psi} = \scalprod{\rho\mu}{\psi} = \int_{\R^d}\psi(x)\rho(x)\mu(\de x),\quad\psi\in\mathrm{C_b}(\R^d).
	\end{equation*}
	From Definition \ref{def: lfdermeas} it holds that $\lfm m (\mu,x) = \rho(x)\delta_x$, where $\delta_x$ is the Dirac measure in $x\in\R^d$. Moreover, for any $\mu\in\Mp$, $x,y\in\R^d$, we have $\lfm^2 m (\mu,x,y) = 0$.
\end{example}

\begin{example}
	Let us consider $f\in \mathrm{C_b}(\R^d;\R^d)$ and let us define the mapping $m = m(\mu)$ as the push-forward measure through $f$, namely
		$\Mp\ni\mu\mapsto m(\mu) = f_\#\mu.$
	We recall that for every $\psi\in \mathrm{C_b}(\R^d)$ it holds
	\begin{equation*}
		\int_{\R^d}\psi(y)f_\#\mu(\de y) = \int_{\R^d}\psi(f(x))\mu(\de x).
	\end{equation*}
	Thus, we have
	 $\lfm m(\mu,x) = \delta_{f(x)}$, where $\delta_{f(x)}$ is the Dirac measure centered in $f(x)$, $x\in\R^d$.
\end{example}

\subsection{Some properties}\label{ssec: someprop}
We list here some properties which will be required in the next section and that help us to understand how these derivatives can be combined.
A first property we need concerns the symmetry of the second-order derivatives.
\begin{proposition}\label{prop: sim2der}
	Let $u\colon\Mp\to\R$ be of class $ \mathrm {C^2_L}(\Mp)$. The following facts hold:
	\begin{enumerate}
		\item[i.] the second-order linear functional derivative is symmetric in the spatial arguments, that is $\lf^2 u (\mu,x,y) = \lf^2 u (\mu,y,x)$ for every $x,y\in\R^d$ and $\mu\in\Mp$;
		\item[ii.] $\mathrm {D}_x(\lf^2 u (\mu,x,y) ) = \lf \left(\diff _\mu u(\cdot,x)\right)(\mu,y)$;
		\item[iii.] $\mathrm {D}_\mu (\diff _\mu u(\cdot, x))(\mu,y) = \diff ^2_\mu u(\mu,x,y)$.
	\end{enumerate}
\end{proposition}
\begin{proof}
	The proof follows the one of \cite[Lemma 2.2.4]{cardaliaguetdelaruelasrylions}, without relevant modifications in the argument.
\end{proof}
\begin{remark}
	Proposition \ref{prop: sim2der} allows us to characterize the second-order $L$-derivative as the $L$-derivative of the mapping $\mu\mapsto \diff _\mu u(\mu,x)$, for every fixed $x\in\R^d$, as we do for the linear functional case. We also notice that the first property in Proposition \ref{prop: sim2der} holds more generally when $u\in \mathrm {C^2}(\Mp)$.
\end{remark}

\begin{remark}
	If we consider functions over $\prob(\R^d)$, it is necessary to add an additional correction to the Schwarz-type identity $i.$ in Proposition \ref{prop: sim2der} (see \cite[Lemma 2.2.4]{cardaliaguetdelaruelasrylions}). In particular it holds that 
	\begin{equation*}
		\lf^2 u(\mu,x,y) = \lf^2 u(\mu,y,x) + \lf u (\mu,x) - \lf u (\mu, y),\quad\mu\in\prob(\R^d),x,y\in\R^d.
	\end{equation*} 
	We can notice that the correction terms disappear if we integrate with respect to $\mu\otimes\mu$.
\end{remark}

The next proposition shows how a to compute the linear functional derivative of the composition of a function from $\R$ to $\R$ and a function from $\Mp$ to $\R$, by a chain rule similar to the classical one.
\begin{proposition}\label{prop: comp}
	Let  $h\in\mathrm{C^1_b}(\R)$ and $g\in\mathrm{C^1}(\Mp)$. Then the composition
	\begin{equation*}
		\Mp\ni\mu\mapsto h(g(\mu))\in\R
	\end{equation*}
	is in  $\mathrm{C^1}(\Mp)$ and the following chain rule holds:
	\begin{equation*}
		\lf h(g(\cdot))(\mu,x) = h'(g(\mu))\lf g(\mu,x),\quad x\in\R^d,\mu\in\Mp.
	\end{equation*}
\end{proposition}

\begin{proof}
	For every $\mu,\mu'\in \Mp$ we have
	\begin{equation*}
		h(g(\mu')) - h(g(\mu)) = \int_0^1 \frac{\partial}{\partial t}h(g(\mu_t))\de t = \int_0^1 h'\left(g(\mu_t))\right)\frac{\partial}{\partial t}g(\mu_t)\de t,
	\end{equation*}
	where $\mu_t = t\mu'+(1-t)\mu$. If we show that
	\begin{equation*}
		\frac{\partial}{\partial t}g(\mu_t) = \int_{\R^d}\lf g(\mu_t, x)[\mu'-\mu](\de x),
	\end{equation*}
	then we are done. For $h$ fixed, since $g\in\mathrm{C}^1(\Mp)$ we can compute the increment
	\begin{align*}
		\frac{1}{h}\left(g(\mu_{t+h}) - g(\mu_{t})\right) = \frac{1}{h}\int_{0}^1\lf g(\tau h \mu' - \tau h\mu +\mu_t, x) h[\mu'-\mu](\de x)\de \tau.
	\end{align*}
	Then by taking $h\to0$ we get $\frac{\partial}{\partial t}g(\mu_t) = \int_{\R^d}\lf g(\mu_t, x)[\mu'-\mu](\de x)$, where we used the dominated convergence theorem and the joint continuity and boundedness of $\lf g$.
\end{proof}
\begin{remark}
	An easy generalization holds for $h(g_1(\mu),\dots,g_n(\mu))$, where $n\in\N$, $h\in\mathrm{C^1_b}(\R^n)$ and $\{g_i\}_{i=1}^n\subset\mathrm{C^1}(\Mp)$. Then $h\in\mathrm{C^1}(\Mp)$ and it holds:
	\begin{equation*}
		\lf h(g_1(\cdot),\dots,g_n(\cdot))(\mu,x) = \sum_{i = 1}^n \partial_ih(g_1(\mu),\dots,g_n(\mu))\lf g_i(\mu,x),\quad x\in\R^d,\mu\in\Mp.
	\end{equation*}
\end{remark}
Another natural result we would like to have is the chain rules for the composition between functions from $\Mp$ to $\R$ and from $\Mp$ to $\Mp$.
\begin{proposition}\label{prop: complflfm} 
Let $m\in\mathrm{\tilde{C}^1}(\Mp)$ and let $g\in\mathrm{C^1}(\Mp)$. Then the composition mapping $\Mp\ni\mu\mapsto g(m(\mu))\in\R$ is in $\mathrm{C^1}(\Mp)$ and it holds:
	\begin{equation*}
		\lf g(m(\cdot))(\mu,x) = \scalprod{\lfm m (\mu,x)}{\lf g(\cdot)(m(\mu))}.
	\end{equation*}
	Moreover, if $m\in\mathrm{\tilde{C}^2}(\Mp)$ and $g\in\mathrm{C^2}(\Mp)$, then $g(m(\cdot))\in\mathrm{C^2}(\Mp)$ with
	\begin{equation*}
		\lf^2 g(m(\cdot))(\mu,x,y) = \scalprod{\lfm^2 m (\mu,x,y)}{\lf g(\cdot)(m(\mu))} + \scalprod{\lfm m (\mu,x)\otimes \lfm m (\mu,y)}{\lf^2 g(\cdot)(m(\mu))}.
	\end{equation*}
\end{proposition}
\begin{proof}
	For every $\mu,\mu'\in \Mp$ we have
	\begin{equation*}
		g(m(\mu')) - g(m(\mu)) = \int_0^1 \frac{\partial}{\partial t}g(m(\mu_t))\de t ,
	\end{equation*}
	where $\mu_t = t\mu'+(1-t)\mu$. Then, by the regularity of $g$ and $m$ follows that
	\begin{equation}\label{eqn: 1}
		\frac{\partial}{\partial t}g(m(\mu_t)) = \int_{\R^d}\scalprod{\lfm m(\mu_t, x)}{\lf g(\cdot)(m(\mu_t))}[\mu'-\mu](\de x),
	\end{equation}
	 and so the thesis. In the same way, one can deduce the result for the second-order derivative.
\end{proof}

Finally, we state a proposition regarding the differentiation of products. We omit the proof since it is analogue to the two above.
\begin{proposition}\label{prop: prodrulelf}
	Let $f,g\in\mathrm{C^1}(\Mp)$. Then the product map 
	\begin{equation*}
		\Mp\ni\mu\mapsto f(\mu)g(\mu)\in\R 
	\end{equation*} is in $\mathrm{C^1}(\Mp)$ and the following product rule holds:
	\begin{equation*}
		(\lf [fg]) (\mu,x) = f(\mu)\lf g(\mu,x) + g(\mu)\lf f(\mu,x). 
	\end{equation*}
	Moreover, if $m\in\mathrm{\tilde{C}^1}(\Mp)$ and if $\psi\colon\R^d\times\Mp\to\R$ is bounded, of class $\mathrm{C^1}(\Mp)$ in the measure argument and continuous in the spatial argument, then the mapping
	\begin{equation*}
		\mu\ni\Mp\mapsto h(\mu)\defeq \scalprod{m(\mu)}{\psi(\cdot,\mu)}\in\R
	\end{equation*}
	is in $\mathrm{C^1}(\Mp)$ and it holds
	\begin{equation*}
		\lf h(\mu,x) = \scalprod{\lfm m (\mu,x)}{\psi(\cdot,\mu)} + \scalprod{m(\mu)}{\lf\psi(\cdot,\mu,x)}.
	\end{equation*}
\end{proposition}

\section{Approximation of real-valued functions over the space of positive measures}\label{sec: approx}
Here we discuss how to approximate real-valued functions over $\Mp$ with a class of simpler functions, which allows easier and explicit computations. 
Our technique is based on a construction which is well known for function over space of probability measures, in particular on $\ptwo$ endowed with the Wasserstein metric. Given $\mu\in\ptwo$, we can introduce its empirical approximation 
\begin{equation}
	\mu_n = \frac{1}{n}\sum_{i=1}^n\delta_{X_i},\quad n\geq 1,
\end{equation} 
where $\{X_i\}_{i=1}^n$ are independent identically distributed (i.i.d.) random variables with law $\mu$ (over an arbitrary probability space $\spprob$). It can be shown, see for instance \cite[Section 5.1.2]{carmonadelarue1}, that $W_2(\mu,\mu_n)\to 0$ almost surely and in $L^2(\Omega)$ as $n\to+\infty$.
If $u\colon\ptwo\to\R$, we can introduce its empirical projection $u^n(\mu)\defeq u(\mu_n)$ and if $u$ is bounded and continuous with respect to $W_2$ one can conclude that, for every $\mu\in\ptwo$, 
\begin{equation}\label{eqn: appcd}
	\E{u^n(\mu)} = \E{u(\mu_n)} = \scalprod{\mu^{\times n}}{u\left(\frac{1}{n}\sum_{i=1}^n\delta_{\cdot_i}\right)}\to u(\mu),\quad n\to+\infty,
\end{equation}
where we used the notation
\begin{equation*}
	\scalprod{\mu^{\times n}}{u\left(\frac{1}{n}\sum_{i=1}^n\delta_{\cdot_i}\right)} = \int_{\R^{dn}} u\left(\frac{1}{n}\sum_{i=1}^n\delta_{x_i}\right) \mu(\de x_1)\dots \mu(\de x_n).
\end{equation*}
We can set $\phi^n(\mu) = \scalprod{\mu^{\times n}}{u\left(\frac{1}{n}\sum_{i=1}^n\delta_{\cdot_i}\right)}$ and thus the family $\{\phi^n\}_{n\geq1}$ approximates pointwise $u$.\\

The goal of our approximation technique is to find a class of functions with good properties that allows us to approximate functions over positive measures together with their derivatives, when they exist.
The first step in our procedure is to adapt the previous argument to a space of finite positive measures.
Let us introduce, for $k> 1$, $\mathcal{M}_{2,k}^+(\R^d)\defeq\{\mu\in\mathcal{M}_2^+(\R^d)\colon \mu(\R^d) \in \left[\frac{1}{k},k\right]\}$. In the end, we will be able to approximate a function $u\colon\mathcal{M}_{2,k}^+(\R^d)\to\R$, where $k>1$ is fixed, in a way that allow us to approximate also its derivatives, when $u\in\mathrm {C_L^2}(\mathcal{M}^+_{2,k}(\R^d))$.  We can define, for every $\mu\in\mathcal{M}_{2,k}^+(\R^d)$,
\begin{equation}
	u^n(\mu) \defeq u\left(\frac{\mu(\R^d)}{n}\sum_{i=1}^n\delta_{X_i}\right),\quad n\geq1,
\end{equation}
where $\{X_i\}_{i=1}^n$ are i.i.d. random variables with law $\mu/\mu(\R^d)$. If we fix $\mu\in\mathcal{M}_{2,k}^+(\R^d)$ and we ask the mapping $\ptwo\ni\pi\mapsto u\left(\mu(\R^d)\pi\right)\in\R$ to be bounded and continuous in $W_2$, then we can conclude, as for \eqref{eqn: appcd}, that
\begin{equation}\label{eqn: approxcilgen}
	\phi^n(\mu)\defeq\E{u^n(\mu)} = \frac{1}{\mu(\R^d)^n}\scalprod{\mu^{\times n}}{u\left(\frac{\mu(\R^d)}{n}\sum_{i=1}^n\delta_{\cdot_i}\right)}\to u\left(\mu(\R^d)\frac{\mu}{\mu(\R^d)}\right) = u(\mu),
\end{equation}
as $n\to+\infty$.
\begin{remark}
	Notice that if $u\colon\mathcal{M}_{2,k}^+(\R^d)\to\R$ is continuous with respect to the weak topology, then the mapping  $\ptwo\ni\pi\mapsto u\left(\mu(\R^d)\pi\right)\in\R$ is continuous in $W_2$ for every fixed $\mu\in\mathcal{M}_{2,k}^+(\R^d)$. Indeed if $W_2(\pi,\pi')\to 0$, then $\scalprod{\mu(\R^d)\pi}{\psi}\to\scalprod{\mu(\R^d)\pi'}{\psi}$ for every $\psi\in\mathrm{C_b}(\R^d)$ and then one conclude thanks to the continuity of $u$.
\end{remark}
\begin{remark}
	It is easy to show that the mapping 
	\begin{equation*}
		\R^{d\times n}\ni(x_1,\dots,x_n)\mapsto \tilde{u}^n(x_1,\dots,x_n)\defeq u\left(\frac{\mu(\R^d)}{n}\sum_{i=1}^n\delta_{x_i}\right)
	\end{equation*} is in $\mathrm {C_b^2}(\R^{d\times n})$ if $u\in \mathrm {C_L^2}(\Mp)$. Indeed, let $h>0$ and $k\in\{1,\dots,n\}$, then
	\begin{multline*}
		\frac{1}{h}\left[ \tilde{u}^n(x_1,\dots,x_k+h,\dots,x_n) -  \tilde{u}^n(x_1,\dots,x_k,\dots,x_n)\right] \\= \frac{\mu(\R^d)}{n}\int_0^1\frac{\lf u \left(m_{\theta,h}^{x_k},x_k+h\right) - \lf u \left(m_{\theta,h}^{x_k},x_k\right)}{h}\de\theta,
	\end{multline*}
	with $m_{\theta,h}^{x_1}\defeq \frac{\mu(\R^d) }{n}\left(\sum_{i\neq k}^n\delta_{x_i} +\theta \delta_{x_k+h} + (1-\theta) \delta_{x_k}\right)$. If we let $h\to0$, we obtain 
	\begin{equation*}
		\partial_k  \tilde{u}^n(x_1,\dots,x_n) = \frac{\mu(\R^d)}{n}\diff _\mu u\left(\frac{\mu(\R^d)}{n}\sum_{i=1}^n\delta_{x_i},x_k\right),
	\end{equation*} 
	which is continuous and bounded thanks to the regularity of $\diff _\mu u$. We can proceed analogously for the second order derivatives. Moreover, if we consider $k> 1$ and $u\in \mathrm {C_L^2}(\mathcal{M}^+_{2,k}(\R^d))$, then the mapping 
	$\R^{d\times n}\times \left[\frac{1}{k}, k\right]\ni(x_1,\dots,x_n,z)\mapsto \bar{u}^n(x_1,\dots,x_n,z)\defeq u\left(\frac{z}{n}\sum_{i=1}^n\delta_{x_i}\right)$ is in $\mathrm{ C_b^2}(\R^{d\times n}\times \left[\frac{1}{k}, k\right])$. Indeed, let $h$ be an admissible increment, then
	\begin{multline*}
		\frac{1}{h}\left[ \bar{u}^n(x_1,\dots,x_n,z+h) -  \bar{u}^n(x_1,\dots,x_n,z)\right] \\
		=\frac{1}{n}\sum_{k=1}^d\int_0^1\lf u\left(\frac{z + \theta h}{n}\sum_{i=1}^d\delta_{x_i},x_k\right)\de \theta\to \frac{1}{n}\sum_{k=1}^d\lf u \left(\frac{z}{n}\sum_{i=1}^d\delta_{x_i},x_k\right),
	\end{multline*}
	as $h\to0$, which is continuous and bounded thanks to the regularity of $\lf u$. As before, we can proceed analogously for the second order derivatives.
\end{remark}
The set $\mathcal{M}^+_{2,k}(\R^d)$ is not compact with respect to the weak topology and in the further computations this will give rise to some problems. So, we want to restrict to the case in which the measures are in a compact subset of $\mathcal{M}^+_{2,k}(\R^d)$. Let us introduce the family of compact rectangles $\{K_N = [-N,N]^d\}_{N\geq1}\subset\R^d$. Then, for every $N\geq1$, we define
\begin{equation*}
	\mathcal{H}_N^k\defeq\left\{\mu\in\mathcal{M}^+_{2,k}(\R^d)\colon \supp \mu\subset K_N\right\},
\end{equation*}  
which is compact in the weak topology. Given $\mu\in\mathcal{M}^+_{2,k}(\R^d)$, we denote with $\rho^N\mu$ the measure such that $\frac{\de \rho^N\mu}{\de \mu}  = \rho^N$, with  $\rho^N$ positive and smooth cut-off function equal to $1$ in $K_N$ and identically zero outside $K_{N+1}$. We notice that $\rho^N\mu$ is always in $\mathcal{H}_{N+1}^k$. Thus, for every $N\geq1$ we  set
\begin{equation}\label{eqn: approxcomp}
	\hat{u}^N(\mu)\defeq u(\rho^N\mu),\quad \mu\in\mathcal{M}^+_{2,k}(\R^d).
\end{equation}
In the following lemma, we show how we can approximate a function $u\colon\mathcal{M}^+_{2,k}(\R^d)\to\R$ with $\{\hat{u}^N\}_{N\geq1}$. 
\begin{lemma}\label{lemma: approxsuppcomp}
Let $k>1$, let $u$ be in $\mathrm {C_L^2}(\mathcal{M}^+_{2,k}(\R^d))$ and let the sequence $\{\hat{u}^N\}_{N\geq1}$ be defined by \eqref{eqn: approxcomp}.
	Then, for every $\mu\in\mathcal{M}^+_{2,k}(\R^d)$, $\hat{u}^N(\mu)\to u(\mu)$ as $N\to+\infty$ and $\{\lf \hat{u}^N\}_{N\geq1}$, $\{\lf^2\hat{u}^N\}_{N\geq1}$, $\{\diff _\mu \hat{u}^N\}_{N\geq1}$, $\{\diff _\mu^2\hat{u}^N\}_{N\geq1}$, $\{\diff _x\diff _\mu\hat{u}^N\}_{N\geq1}$, $\{\diff _x\lf^2\hat{u}^N\}_{N\geq1}$ pointwise converge to the respective derivatives of $u$. Moreover, $\norm{\hat{u}^N}_\infty\leq\norm{u}_\infty$, and there exists $C>0$ independent of $u,N,k$ such that 
	\begin{align*}
	\norm{\lf \hat{u}^N}&\leq\norm{\lf u},&\norm{\lf^2\hat{u}^N}\leq\norm{\lf^2u},\\
	\norm{\diff _\mu \hat{u}^N}&\leq C(\norm{\diff _\mu u}+\norm{\lf u}),& \norm{\partial_x\diff _\mu \hat{u}^N}\leq C(\norm{\diff _\mu u}+\norm{\lf u} + \norm{\diff _x\diff _\mu u}),\\
	\norm{\diff _x\lf^2 \hat{u}^N}&\leq C(\norm{\lf^2 u}+\norm{\diff _x\lf^2 u}), &\norm{\diff _\mu^2 \hat{u}^N}\leq C(\norm{\lf^2 u}+\norm{\diff _x\lf^2 u}+\norm{\diff _\mu^2 u}),
\end{align*}
	where the norms are meant as $\norm{\cdot}_\infty$, that is the supremum norm over all the arguments of the function.
\end{lemma}
\begin{proof}
First, we notice that $\rho^N\mu\to\mu$ weakly as $N\to+\infty$. Indeed, for every $\varphi \in C_b(\R^d)$, we have
\begin{equation*}
	\scalprod{\rho^N\mu}{\varphi} = \int_{\R^d}\rho^N(x)\varphi(x)\mu(\de x)\to\int_{\R^d}\varphi(x)\mu(\de x),
\end{equation*}
as $n\to+\infty$, thanks to dominated convergence. Thus, since $u$ is continuous with respect to the weak topology, we have that for every $\mu\in\mathcal{M}^+_{2,k}(\R^d)$, $\hat{u}^N(\mu)\to u(\mu)$ as $N\to+\infty$. Moreover, it is immediate that $\norm{\hat{u}^N}_\infty\leq\norm{u}_\infty$.\\\\
Regarding the linear functional derivatives, from Proposition \ref{prop: complflfm} with $m(\mu)= \rho^N\mu$ and $g(\mu) = u(\mu)$, it follows that for every $\mu\in\mathcal{M}^+_{2,k}(\R^d)$ and $x,y\in\R^d$
\begin{align*}
	\lf \hat{u}^N(\mu, x) & = \rho^N(x)\lf u (\rho^N\mu,x)\to\lf u (\mu,x), &\mathrm{as}\ N\to+\infty,\\
	\lf^2 \hat{u}^N(\mu,x,y)& = \rho^N(x)\rho^N(y)\lf^2 u (\rho^N\mu,x,y)\to\lf^2 u(\mu,x,y), &\mathrm{as}\ N\to+\infty,
\end{align*}
thanks to the continuity of $\lf u$ and $\lf ^2u$ and the fact that $\rho^N(x)\to 1$ as $N\to+\infty$, for every $x\in\R^d$. As before, the estimates on the norms easily follows.\\\\
To conclude, it remains to show the convergence of the derivatives in space of $\lf u$ and $\lf^2u$. For instance, we have that, for every $\mu\in\mathcal{M}^+_{2,k}(\R^d)$ and $x\in\R^d$
\begin{multline*}
	\diff _\mu \hat{u}^N(\mu,x) = \diff _x\lf \hat{u}^N(\mu,x) \\ = \diff _x\rho^N(x)\lf u (\rho^N\mu,x) + \rho^N(x)\diff _\mu u (\rho^N\mu,x)\to\diff _\mu u(\mu,x),\quad N\to+\infty.
\end{multline*}
Indeed, the first summand tends to 0 since $\supp \diff _x\rho^N\subset[N,N+1]$ and $\lf u$ is bounded, whilst the second one tends to $\diff _\mu u(\mu,x)$ thanks to the continuity of $\diff _\mu u$ and the fact that $\rho^N(x)\to 1$ as $N\to+\infty$, for every $x\in\R^d$. In a similar way, we get
\begin{align*}
	 \diff _x\diff _\mu \hat{u}^N(\mu,x)  =& \diff _x^2\rho^N(x)\lf u (\rho^N\mu,x) \\ & \qquad+ 2 \diff _x\rho^N(x)\diff _\mu u (\rho^N\mu,x)  + \rho^N(x) \diff _x\diff _\mu u (	\rho^N\mu,x),\\\\
	 \diff _x\lf^2 \hat{u}^N(\mu,x,y)  =&\rho^N(y) \diff _x\rho^N(x)\lf^2 u (\rho^N\mu,x,y) + \rho^N(y)\rho^N(x) \diff _x\lf^2 u (\rho^N\mu,x,y),\\\\
	\diff _\mu^2\hat{u}^N(\mu,x,y)  =& \diff _x\rho^N(x) \diff _y\rho^N(y)^\top\lf^2 u (\rho^N\mu,x,y) \\ &\qquad+\rho^N(y) \diff _x\rho^N(x) \diff _y\lf^2 u (\rho^N\mu,x,y) ^\top\\
	&\qquad\qquad+\rho^N(x) \diff _x\lf^2 u (\rho^N\mu,x,y) \diff _y \rho^N(y)^\top  \\&\qquad\qquad\quad+ \rho^N(y)\rho^N(x)\diff _\mu^2 u (\rho^N\mu,x,y),
\end{align*}
and then we can prove the convergence of the remaining derivatives as before. \\\\
Regarding the estimates on the norm, it follows that 
\begin{align*}
	\norm{\diff _\mu \hat{u}^N}&\leq C(\norm{\diff _\mu u}+\norm{\lf u}),\quad& \norm{\partial_x\diff _\mu \hat{u}^N}\leq C(\norm{\diff _\mu u}+\norm{\lf u} + \norm{\partial_x\diff _\mu u}),\\
	\norm{\partial_x\lf^2 \hat{u}^N}&\leq C(\norm{\lf^2 u}+\norm{\partial_x\lf^2 u}),\quad &\norm{\diff _\mu^2 \hat{u}^N}\leq C(\norm{\lf^2 u}+\norm{\partial_x\lf^2 u}+\norm{\diff _\mu^2 u}),
\end{align*}
where $C = \max\left\{1, 2\norm{\partial_x\rho^N},\norm{\partial_x^2\rho^N}, \norm{\partial_x\rho^N(\partial_y\rho^N)^\top}\right\}$ and all the norms are meant as $\norm{\cdot}_\infty$. Notice that $C$ can be chosen independent of $N$, $k$ and $u$, thanks to the particular structure of $\rho^N$.
\end{proof}
\noindent Thanks to Lemma \ref{lemma: approxsuppcomp}, we can approximate functions in $\mathrm {C_L^2}(\mathcal{M}^+_{2,k}(\R^d))$ with functions in $\mathrm {C_L^2}(\mathcal{H}_N^k)$. Since we need more regular approximants, our goal now is to show that if $u\in\rm C_L^2(\mathcal{H}_N^k)$, with $k>1$ and $N\geq 1$ fixed, then its derivatives can be approximated by the corresponding derivatives of the sequence $\{\phi^n\}_{n\geq1}$ introduced in \eqref{eqn: approxcilgen}, namely
\begin{equation*}
	\phi^n(\mu)\defeq\E{u^n(\mu)} = \frac{1}{\mu(\R^d)^n}\scalprod{\mu^{\times n}}{u\left(\frac{\mu(\R^d)}{n}\sum_{i=1}^n\delta_{\cdot_i}\right)},\quad \mu\in\mathcal{H}_N^k.
\end{equation*} 

\begin{lemma}\label{lemma: approxcilgen}
	Let $u$ be in $\rm C_L^2(\mathcal{H}_N^k)$, for $k>1$ and $N\geq1$ fixed, and let $\{\phi^n\}_{n\geq1}$ be defined by \eqref{eqn: approxcilgen}. 
	Then, for every $\mu\in\mathcal{H}_N^k$, $\phi^n(\mu)\to u(\mu)$ as $n\to+\infty$ and $\{\lf \phi^n\}_{n\geq1}$, $\{\lf^2\phi^n\}_{n\geq1}$, $\{\diff _\mu \phi^n\}_{n\geq1}$, $\{\diff _\mu^2\phi^n\}_{n\geq1}$, $\{\diff _x\diff _\mu\phi^n\}_{n\geq1}$, $\{\diff _x\lf^2\phi^n\}_{n\geq1}$ converge pointwise to the respective derivatives of $u$. Moreover, $\norm{\phi^n}_\infty\leq\norm{u}_\infty$ and the same holds for the derivatives, up to a multiplicative constant independent of $u$, $N$ and $k$.
\end{lemma}
\begin{remark} \label{rmk: cilgen}
Thanks to Lemma \ref{lemma: approxcilgen}, we can approximate functions in $\mathrm {C_L^2}(\mathcal{H}_N^k)$ with functions $\phi^n\colon\mathcal{H}_N^k\to\R$ of the form $\phi^n(\mu) = \scalprod{\frac{\mu^{\times n}}{\mu(\R^d)^n}}{\varphi_n(\cdot,\dots,\cdot,\mu(\R^d))}$, where $\varphi_n\in \mathrm {C_b^2}(K_N^n\times \left[\frac{1}{k},k\right])$ is symmetric in the first $n$ arguments and $K_N = [-N,N]^d$.
\end{remark}
\begin{proof}
First, $u\colon\mathcal{H}_N^k\to\R$ is absolutely continuous with respect to the weak topology and so it is also for the mapping $\mathcal{P}_2(K_N)\ni\pi\mapsto u(\mu(\R^d)\pi)$, with $\mu\in\mathcal{H}_N^k$. Moreover, since we are considering measures over a compact subset of $\R^d$, the absolute continuity of this last mapping also holds in $W_2$. The same consideration is true also for $\lf u$, $\diff _\mu u$ and $\mathrm D_x\diff _\mu u$ (resp.  $\lf^2 u$, $\mathrm D_x\lf^2 u$ and $\mathrm D^2_\mu u$) when restricted to $\mathcal{H}_N^k\times K_N$ (resp. $\mathcal{H}_N^k\times K_N\times K_N$). \\\\
We already showed the convergence of $\phi^n(\mu)$ to $u(\mu)$ for every $\mu\in\mathcal{H}_N^k$. Let us discuss the convergence of the linear functional derivatives of $\phi^n$. First, we notice that $\phi^n(\mu) = f(\mu)g(\mu)$ with
\begin{equation*}
	g(\mu):=\frac{1}{\mu(\R^d)^n},\quad f(\mu):= \scalprod{\mu^{\times n}}{u\left(\frac{\mu(\R^d)}{n}\sum_{i=1}^n\delta_{\cdot_i}\right)},\quad \mu\in\mathcal{H}_N^k.
\end{equation*}
From Proposition \ref{prop: comp} it follows that $\lf g(\mu,x) = -\frac{n}{\mu(\R^d)^{n+1}}$,
whilst from Proposition \ref{prop: prodrulelf} combined with \ref{prop: complflfm} we get
\begin{equation*}
	\lf f(\mu,x) = n\scalprod{\mu^{\times(n-1)}}{u\left(\frac{\mu(\R^d)}{n}\sum_{i=1}^{n-1}\delta_{\cdot_i}+\frac{\mu(\R^d)}{n}\delta_x\right)} + \frac{1}{n}\sum_{j=1}^n\scalprod{\mu^{\times n}}{\lf u\left(\frac{\mu(\R^d)}{n}\sum_{i=1}^n\delta_{\cdot_i},\cdot_j\right)}.
\end{equation*}

Thus, from Proposition \ref{prop: prodrulelf},  we have that for $(\mu,x)\in\mathcal{H}_N^k\times K_N$, 
\begin{multline*}
	\lf \phi^n(\mu,x) \\= \frac{n}{\mu(\R^d)^n}\scalprod{\mu^{\times(n-1)}}{u\left(\frac{\mu(\R^d)}{n}\sum_{i=1}^{n-1}\delta_{\cdot_i}+\frac{\mu(\R^d)}{n}\delta_x\right)}
	-\frac{n}{\mu(\R^d)^{n+1}}\scalprod{\mu^{\times n}}{u\left(\frac{\mu(\R^d)}{n}\sum_{i=1}^n\delta_{\cdot_i}\right)}\\
	+ \frac{1}{n}\sum_{j=1}^n\scalprod{\left(\frac{\mu}{\mu(\R^d)}\right)^{\times n}}{\lf u\left(\frac{\mu(\R^d)}{n}\sum_{i=1}^n\delta_{\cdot_i},\cdot_j\right)},
\end{multline*}
which can be written by using the random variables $\{X_i\}_{i=1}^n$ as
\begin{multline}\label{eqn: lfapproxcilgen}
	\lf \phi^n(\mu,x) = \frac{n}{\mu(\R^d)}\left\{\E{u\left(\frac{\mu(\R^d)}{n}\sum_{i=1}^{n-1}\delta_{X_i}+\frac{\mu(\R^d)}{n}\delta_x\right)} - \E{u\left(\frac{\mu(\R^d)}{n}\sum_{i=1}^{n}\delta_{X_i}\right)}\right\}\\
	+ \frac{1}{n}\sum_{j=1}^n \E{\lf u \left(\frac{\mu(\R^d)}{n}\sum_{i=1}^{n}\delta_{X_i},X_j\right)}.
\end{multline}
Regarding the first two terms of \eqref{eqn: lfapproxcilgen}, since $u$ is differentiable in linear functional sense, we get that their difference is equal to
\begin{multline*}
		\E{\int_0^1\lf u \left(\frac{\mu(\R^d)}{n}\sum_{i=1}^{n-1}\delta_{X_i} + \theta\frac{\mu(\R^d)}{n}\delta_{x} + (1-\theta)\frac{\mu(\R^d)}{n}\delta_{X_n},x\right)\de\theta}\\
		 - \E{\int_0^1\lf u \left(\frac{\mu(\R^d)}{n}\sum_{i=1}^{n-1}\delta_{X_i} + \theta\frac{\mu(\R^d)}{n}\delta_{x} + (1-\theta)\frac{\mu(\R^d)}{n}\delta_{X_n},X_n\right)\de\theta},
\end{multline*}
which tends to $\lf u(\mu,x) - \scalprod{\frac{\mu}{\mu(\R^d)}}{\lf u (\mu,\cdot)}$, thanks to dominated convergence and the regularity of the linear functional derivative. On the other hand, the last therm in \eqref{eqn: lfapproxcilgen} converges to $\scalprod{\frac{\mu}{\mu(\R^d)}}{\lf u (\mu,\cdot)}$ (which coincide, for instance, with $\E{\lf u (\mu,X_1)}$). Indeed, 
\begin{multline*}
	\left\vert\frac{1}{n}\sum_{j=1}^n \E{\lf u \left(\frac{\mu(\R^d)}{n}\sum_{i=1}^{n}\delta_{X_i},X_j\right)} -  \E{\lf u (\mu,X_1)}\right\vert\\
	\leq\left\vert\frac{1}{n}\sum_{j=1}^n \E{\lf u \left(\frac{\mu(\R^d)}{n}\sum_{i=1}^{n}\delta_{X_i},X_j\right)} - \frac{1}{n} \E{\sum_{j=1}^n\lf u (\mu,X_j)}\right\vert\\
	+ \left\vert\frac{1}{n} \E{\sum_{j=1}^n\lf u (\mu,X_j)} - \E{\lf u (\mu,X_1)}\right\vert,
\end{multline*}
then the first term tends to zero thanks to the uniform continuity of $\lf u$   
and the second one tends to zero due to the law of large numbers. From \eqref{eqn: lfapproxcilgen} it also follows that $\norm{\lf\phi^n}_\infty\leq 3\norm{\lf u}_\infty$. Similarly we are able to show the pointwise convergence of $\lf^2\phi^n$ and the estimate for $\norm{\lf^2\phi^n}_\infty$.\\

Now we study the convergence of the $L$-derivatives of $\phi^n$, by differentiating \eqref{eqn: lfapproxcilgen} with respect to $x$. The second and the third term do not depend on $x$, whilst for the first one we can compute the increment:
\begin{equation*}
	\frac{1}{h}\left[\lf\phi^n(\mu,x+he_k) - \lf\phi^n(\mu,x)\right] = \E{\int_0^1\frac{\lf u \left(m^x_{\theta,h},x+he_k\right) - \lf u \left( m^x_{\theta,h},x\right)}{h}\de \theta},
\end{equation*}
where $h\in\R$, $k\in\{1,\dots,d\}$, $e_k$ is the $k$-th vector of the canonical basis of $\R^d$ and 
\begin{equation*}
	m^x_{\theta,h}\defeq  \frac{\mu(\R^d)}{n}\left(\sum_{i=1}^{n-1}\delta_{X_i} + \theta\delta_{x+he_k} +(1-\theta)\delta_{x}\right).
\end{equation*} Then we can pass to the limit as $h\to 0$ and thanks to the uniform continuity and the differentiability of $\lf u$ we get
\begin{equation*}
	\diff _\mu \phi^n(\mu,x) = \diff _x\lf \phi^n(\mu,x) = \E{\mathrm D_x\lf u\left(\frac{\mu(\R^d)}{n}\sum_{i=1}^{n-1}\delta_{X_i} + \frac{\mu(\R^d)}{n}\delta_{x} ,x\right)} ,
\end{equation*}
which tends to $\diff _\mu u(\mu,x)$ as $n\to +\infty$ and from which we can also deduce $\norm{D_mu\phi^n}_\infty\leq\norm{\diff _\mu u}_\infty$. Again, with similar computations one can show the pointwise convergence and the norms estimates for $\{\diff _x\lf^2\phi^n\}_{n\geq1}$, $\{\diff _x\diff _\mu\phi^n\}_{n\geq1}$, $\{\diff _\mu^2\phi^n\}_{n\geq1}$.\\
\end{proof}
A further approximation can be achieved by substituting $\varphi$ in Remark \ref{rmk: cilgen} with a certain family of symmetric polynomials. The obtained function has a more regular structure and is easier to study. We call functions of this type cylindrical and they represent the last step in our approximation.
\begin{definition}[Cylindrical functions]\label{def: cyl}
We say that $u: \Mp \to \R$ is a cylindrical function of order $k \in \N$ if there exist $n \in \N$, $g \in \mathrm {C^k_b}(\R^n)$ and $\{\psi_i\}_{i=1}^n\subset \mathrm{C}^k_\mathrm {b}(\R^d)$ such that
\begin{equation*}
u(\mu) = g\left(\scalprod{\mu}{\psi_1},\dots,\scalprod{\mu}{\psi_n}\right), \quad \forall \, \mu \in \Mp.
\end{equation*}
We denote with $\mathcal{C}_k(\Mp)$ the set of cylindrical functions of order $k$. 
\end{definition}
In order to study the convergence in the next Lemma, we consider the following norm: let $N\geq1$, $k>1$ and let $u\in \mathrm {C_L^2}(\mathcal{H}_N^k)$
, then
\begin{multline}\label{eqn: c2lnrom}
	\norm{u}_{\mathrm {C_L^2}(\mathcal{H}_N^k)}  \defeq \sup_{\mu\in\mathcal{H}_N^k, \atop x,y\in K_N}\big( \vert u(\mu)\vert + \vert\lf u(\mu,x)\vert + \vert\lf^2u(\mu,x,y) \vert+ \vert\mathrm D_x\lf^2 u(\mu,x,y)\vert \\+ \vert \diff _\mu u(\mu,x)\vert +  \vert \mathrm D_x\diff _\mu u(\mu,x)\vert + \vert \mathrm D^2_\mu u(\mu,x,y)\vert \big),
\end{multline}
where we recall that $K_N = [-N,N]^d$.
\begin{lemma}\label{lemma: approxcilgencil}
	Let $N\geq 1$ and $k>1$ be fixed. Let $f\colon\mathcal{H}_N^k\to\R$ be defined, for $r\in\N$, as $f(\mu)\defeq\scalprod{\frac{\mu^{\times r}}{\mu(\R^d)^r}}{\varphi(\cdot,\dots,\cdot,\mu(\R^d))}$, where $\varphi\in \mathrm {C_b^2}\left(K_N^r\times \left[\frac{1}{k},k\right]\right)$ is symmetric in the first $r$ arguments. Then, there exists a sequence $\{f^n\}_{n\geq1}\subset\mathcal{C}_2(\Mp)$ such that $\norm{f - f^n}_{\rm C_L^2(\mathcal{H}_N^k)}\to 0$ as $n\to+\infty$. 
\end{lemma}
\begin{proof}
Since $K_N^r\times \left[\frac{1}{k},k\right]\subset\R^{dr + 1}$ is compact, we can find a family of symmetrical polynomials $\{\varphi_n\}_{n\geq1}$ which approximates $\varphi$ in $\rm C^2$ norm (we can choose a symmetric version of Bernstein approximants, see for instance \cite[Section 3.2]{bucurpaltineanu}). More precisely,
\begin{equation*}	
	\varphi_n(x_1,\dots,x_r,z) = \sum_{i=1}^{\ell(n)}h_i(z)\prod_{j=1}^rg_{i,j}(x_j),
\end{equation*}
where $\ell(n)\in\N$, $g_{i,j}\colon K_N\to\R$ and $h_i\colon \left[\frac{1}{k},k\right]\to\R$ are monomials, for every $i=1,\dots,\ell(n)$ and $j=1,\dots,r$. Let us introduce $f^n(\mu)\defeq \scalprod{\frac{\mu^{\times r}}{\mu(\R^d)^r}}{\varphi_n(\cdot,\dots,\cdot,\mu(\R^d)}$, for every $n\geq1$. This sequence is in $\mathcal{C}_2(\Mp)$, indeed, for every $n\geq1$, \begin{equation*}
	f^n(\mu) 
	= g^n(\scalprod{\mu}{g_{1,1}},\dots,\scalprod{\mu}{g_{\ell(n),r}},\mu(\R^d)),
\end{equation*}
where $g^n$ is defined over $\R^{dr\ell(n)}\times \R$ as
\begin{equation*}
	g^n(\xi_{1,1},\dots,\xi_{\ell(n),r},\zeta) = \sum_{i=1}^{\ell(n)}\zeta^{e(i)-1}\prod_{j=1}^r \xi_{i,j},
\end{equation*}
where we denoted with $e(i)$ the exponents in the monomials $h_i$. This function $g^n$ is symmetric, twice continuously differentiable but not bounded. However, we can consider the product with a smooth symmetric (in the first $r$ variables) cut-off function, equal to $1$ in the rectangle $[-R-1,R+1]^{dr\ell(n)}\times\left[\frac{1}{k},k\right]$, where $R\defeq\max\{|\scalprod{\mu}{g_{1,1}}|,\dots,|\scalprod{\mu}{g_{\ell(n),r}}|\}$, and vanishing smoothly outside. This function, that we denote again with $g^n$ for simplicity, is symmetric and in $\rm C_b^2(\R^{dr\ell(n)}\times\R)$. We can notice that at this point we needed the fact that $\mu(\R^d) \ge 1/k$, which was included in the definition of $\mathcal{H}^k_N$, and not only $\mu(\R^d) > 0$. Indeed, we need  $\varphi\in \mathrm {C_b^2}\left(K_N^r\times \left[\frac{1}{k},k\right]\right)$ in order to  introduce the sequence $\{\varphi_n\}_{n\geq1}$ and consequently $\{f^n\}_{n\geq 1}$ .\\\\ 
Let us study the convergence of $\{f^n\}_{n\geq1}$ in norm $\norm{\cdot}_{\mathrm {C^2_L}(\mathcal{H}_N^k)}$. First, $\sup_{\mu\in\mathcal{H}_N^k}\left\vert f^n(\mu) - f(\mu)\right\vert$ tends to $0$ as $n\to+\infty$, thanks to the uniform convergence of $\{\varphi^n\}_{n\geq1}$ and the bound on $\mu(\R^d)$. Regarding the first-order linear functional derivative, we have that for every $\mu\in\mathcal{H}_N^k$ and $x\in K_N$,
\begin{multline}\label{eqn: approlfcilgencil}
	\lf f^n(\mu,x) =r\scalprod{\frac{\mu^{\times(r-1)}}{\mu(\R^d)^r}}{\varphi_n(\cdot_1,\dots,\cdot_{r-1},x,\mu(\R^d))} + \scalprod{\frac{\mu^{\times r}}{\mu(\R^d)^r}}{\partial_z\varphi^n(\cdot,\dots,\cdot,\mu(\R^d))}\\-r\scalprod{\frac{\mu^{\times r}}{\mu(\R^d)^{r+1}}}{\varphi^n(\cdot,\dots,\cdot,\mu(\R^d))}.
\end{multline}
Due to dominated convergence, expression \eqref{eqn: approlfcilgencil} converges as $n\to+\infty$ to the same one with $\varphi$ instead of $\varphi^n$, which is $\lf f$. Analogously, $\lf^2 f^n(\mu,x,y)\to\lf f^n(\mu,x,y)$ as $n\to+\infty$ for every $\mu\in\mathcal{H}_N^k$ and $x,y\in K_N$.
Moreover, since $\varphi_n$ and its derivatives up to order two converge uniformly to $\varphi$ and its derivatives, we have that
\begin{equation*}
	\sup_{\mu\in\mathcal{H}_N^k, \atop x\in K_N}\left\vert \lf f^n(\mu,x) - \lf f(\mu,x)\right\vert\to0, \qquad \sup_{\mu\in\mathcal{H}^k_N, \atop x,y\in K_N}\left\vert \lf^2 f^n(\mu,x,y) - \lf^2 f(\mu,x,y)\right\vert\to0,
\end{equation*}
as $n\to+\infty$. To conclude the discussion on the convergence, we are able to bring the spatial derivative inside the integral in the first row of \eqref{eqn: approlfcilgencil} and so for every $\mu\in\mathcal{H}_N^k$ and $x\in\R^d$
\begin{multline*}
		\diff _\mu f^n(\mu,x) =r\scalprod{\frac{\mu^{\times(r-1)}}{\mu(\R^d)^r}}{\mathrm D_x\varphi_n(\cdot_1,\dots,\cdot_{r-1},x,\mu(\R^d))}\\\to r\scalprod{\frac{\mu^{\times(r-1)}}{\mu(\R^d)^r}}{\mathrm D_x\varphi(\cdot_1,\dots,\cdot_{r-1},x,\mu(\R^d))} = \diff _\mu f(\mu,x),
\end{multline*}
as $n\to+\infty$. As before, the convergence takes place also uniformly since the first-order derivative of $\varphi_n$ converges uniformly to the one of $\varphi$. In the same way we can show the uniform convergence over $\mathcal{H}_N^k\times K_N$ (or $\mathcal{H}_N^k\times K_N\times K_N$) of $\{\diff _\mu^2 f^n\}_{n\geq1}$, $\{\diff _x\diff _\mu f^n\}_{n\geq1}$, $\{\diff _x\lf^2 f^n\}_{n\geq1}$ to $\diff _\mu^2 f$, $\diff _x\diff _\mu f$, $\diff _x\lf^2f$ respectively. 
\end{proof}
\begin{remark}\label{rmk: lemmasprob}
	We can notice that both the results of Lemma \ref{lemma: approxcilgen} and Lemma \ref{lemma: approxcilgencil} work for functions over $\mathcal{P}_2\left(K_N\right)$, and the approximants remain well defined over $\mathcal{P}_2\left(K_N\right)$. Regarding Lemma \ref{lemma: approxsuppcomp}, we can not simply cut the measure support, since the mass must be kept normalized. In this case, we can approximate $\mu\in\ptwo$ with the family of probabilities obtained by concentrating all the mass outside the compact $K_N$ in the origin, namely $\{\lambda^N_1\mu + \lambda^N_2\delta_0\}_{N\geq1}$, where $\lambda^N_1$ and $\lambda^N_2$ are smooth positive cut-off functions, $\lambda^N_1 = 0$ in $K_{N+1}$,  $\lambda^N_2 = 0$ in $K_{N}^\mathsf{c}$ and $\lambda^N_1+\lambda^N_2 = 1$.  This sequence converges weakly to $\mu\in\ptwo$ and the approximation properties in Lemma \ref{lemma: approxsuppcomp} still hold for $\{\hat{u}^N\}_{N\geq1}\defeq \{u(\lambda^N_1\mu + \lambda^N_2\delta_0)\}_{N\geq1}$.  
\end{remark}
\begin{remark}
	From the proofs of Lemma \ref{lemma: approxsuppcomp}, we can see that if we ask for $u\in \rm C^2(\mathcal{M}^+_{2,k}(\R^d))$, then the approximation for $u$ and its first- and second-order linear functional derivatives still holds. The same applies also for Lemma \ref{lemma: approxcilgen}.
\end{remark}
\begin{remark}
	At the beginning of this section we used the fact that a function $u$ on $\Mp$ can be regarded as a function defined on $\R_+\times\prob(\R^d)$ through the formula
\begin{equation*}
	v(\lambda,\pi) = u(\lambda\pi),\quad\lambda>0, \ \pi\in\prob(\R^d),
\end{equation*}
and conversely
\begin{equation*}
	u(\mu) = v(\mu(\R^d),\frac{\mu}{\mu(\R^d)}),\quad \mu\in\Mp.
\end{equation*}
Probably this fact could be used to deduce some of the results we proved for functions over $\Mp$ (see for instance Section \ref{ssec: someprop}) directly from earlier results for functions over $\prob(\R^d)$. However, we decided to give the proofs directly in the context of $\Mp$ since in this case the computations are more straightforward and since this is the natural framework to study the Zakai equation.
\end{remark}

\section{Two equations related to nonlinear filtering}\label{sec: eqnfilter}
In this section we introduce the filtering equations. Our main results in later sections concern their associated Kolmorogov equations. Here we introduce notation and basic assumptions and we recall well-posedness results and first properties, as well as the connection with the filtering problem.
The nonlinear filtering equations, namely the Zakai equation and the Kushner-Stratonovich equation, arise naturally when the problem of nonlinear filtering is studied. We will use the approach adopted by Szpirglas in \cite{szpirglas} and later by Heunis and Lucic in \cite{heunislucic}, where existence, uniqueness and regularity properties of the equations are proved under appropriate assumptions, without direct reference to the filtering problem. \\

Let us consider a finite time interval $[0,T]$, a complete probability space $\spprob$ with a filtration $\{\mathcal{F}_t\}_{t\in[0,T]}$ which satisfies the usual conditions. Let $W = \{W_t,t\in[0,T]\}$ and $B = \{B_t,t\in[0,T]\}$ be two independent $d$-dimensional $\{\mathcal{F}_t\}$-Brownian motions (we take them with the same dimension for simplicity) and let us consider the $d$-dimensional process $X = \{X_t,t\in[0,T]\}$, called signal process, defined by
\begin{equation}\label{eqn: signal}
	\de X_t = f(X_t)\de s + \sigma(X_t)\de W_t + \bar{\sigma}(X_t)\de B_t,\quad X_0\in L^2(\Omega,\mathcal{F}_0), 
\end{equation}
where the Borel measurable mappings $f\colon\R^d\to\R^d$, $\sigma\colon\R^{d}\to\R^{d\times d}$ and $\bar{\sigma}\colon\R^{d}\to\R^{d\times d}$ will be chosen later in order to have existence of a strong solutions, uniqueness (up to indistinguishability), continuity and Markovianity of $X$. The idea behind the stochastic filtering is that we can not observe directly the signal, but we can only observe a process $Y=\{Y_t,t\in[0,T]\}$, called observation process, defined by
\begin{equation}\label{eqn: obs}
	Y_t = \int_0^t h(X_s)\de s + B_t,
\end{equation}
where $h\colon\R^d\to\R^d$ is a Borel measurable mapping such that $\E{\int_0^T |h(X_s)|^2\de s}<\infty$. If we introduce the observation filtration $\{\mathcal{F}_t^Y\}_{t\in[0,T]}$, where $\mathcal{F}_t^Y$ is the completion with respect to the $\mathbb{P}$-null sets of the $\sigma$-algebra generated by $Y$ up to the time $t$, we can say that the filtering problem is to find a $\prob(\R^d)$-valued process $\Pi = \{\Pi_t,t\in[0,T]\}$, called filter, such that 
\begin{equation}\label{eqn: filtrel}
	\Pi_t(\varphi)=\E{\varphi(X_t)\vert\mathcal{F}^Y_{t}}, 
\end{equation} almost surely, for every $t\in[0,T]$ and $\varphi\in\mathrm{B_b}(\R^d)$. Due to the presence of the observation noise also in the signal equation \eqref{eqn: signal}, we will refer to this problem as stochastic filtering with correlated noise, in contrast with the problem in which $\bar\sigma$ is null, called without correlated noise.
Now, let us state the assumptions on the coefficients necessary for our discussion. They are not always necessary together (see Remark \ref{rmk: lesshp}), but we group them for simplify the exposition.
\begin{hypotheses}\label{hp: generali}
 All the mappings $f,\sigma,\bar{\sigma},h$ are taken Borel-measurable. Moreover we assume:
	\begin{enumerate}
		\item[a.] the mappings $f\colon\R^d\to\R^d$, $\sigma\colon\R^{d}\to\R^{d\times d}$ and $\bar{\sigma}\colon\R^{d}\to\R^{d\times d}$ are Lipschitz continuous;
		\item[b.] the mapping $a:=\sigma\sigma^\top\colon\R^d\to\R^{d\times d}$ is uniformly elliptic, that is there exists $\lambda>0$ such that $\sum_{i,j=1}^d a_{ij}\xi_i\xi_j\geq\lambda\lvert\xi\rvert^2$ for every $x,\xi\in\R^d$;
		\item[c.] the mappings $f,\sigma,\bar\sigma,h$ are bounded.
	\end{enumerate}
\end{hypotheses}

\noindent Under these conditions, the signal is a uniquely characterized Markov process and moreover there exists a $\prob(\R^d)$-valued $\{\mathcal{F}^Y_{t}\}$-optional process which satisfies \eqref{eqn: filtrel}. 

Let us introduce two differential operator, that are defined for every $\psi\in\mathrm{C^2_b}(\R^d)$ by
\begin{equation}\label{eqn: opdiff}
\begin{aligned}
	A\psi(x) &\defeq \sum_{i=1}^d f_i(x)\partial_i \psi(x) + \frac{1}{2}\sum_{i,j=1}^d\left\{\left(\sigma\sigma^\top\right)_{ij}(x)\partial_{ij}\psi(x)+\left(\bar\sigma\bar\sigma^\top\right)_{ij}(x)\partial_{ij}\psi(x)\right\},\quad x\in\R^d,\\
	B_k\psi(x) &\defeq \sum_{i=1}^d \bar{\sigma}_{ik}(x)\partial_i\psi(x) ,\quad x\in\R^d,k=1,\dots d.
\end{aligned}
\end{equation}
It has been proved (see for instance \cite[Chapter 3]{baincrisan}) that the process $\Pi$ satisfies the following stochastic differential equation, called Kushner-Stratonovich equation (or Fujisaki-Kallianpur-Kunita):
\begin{equation}\label{eqn: myfirstks}
	\Pi_t(\psi) = \pi(\psi) + \int_0^t \Pi_s(A\psi)\de s + \int_0^t \left( \Pi_s(h \psi+ B\psi) - \Pi_s(\psi)\Pi_s(h)\right)\cdot\de I_s,\quad \pi =\mathcal{L}(X_0),
\end{equation}
for every $\psi\in\mathrm{C^2_b}(\R^d)$, where $I_t\defeq Y_t-\int_0^t\Pi_s(h)\de s$, $t\in [0,T]$ is called innovation process and it is a $d$-dimensional $\{\mathcal{F}_t^Y\}$-Brownian motion. Another essential result in this framework is that the equation \eqref{eqn: myfirstks} can be rephrased into a linear equation for the evolution of the unnormalized law of the filter, namely the $\mathcal{M}^+(\R^d)$-valued process $\rho = \{\rho_t,t\in[0,T]\}$. In particular, one can prove that the process $\rho$ satisfies the so-called Zakai equation:
\begin{equation}\label{eqn: mufirstzakai}
	\rho_t(\psi) = \pi(\psi) + \int_0^t \rho_s(A\psi)\de s + \int_0^t \rho_s(h \psi+ B\psi)\cdot\de Y_s,\quad \psi \in\mathrm{C^2_b}(\R^d),
\end{equation}
with $\rho_t(\R^d)>0$ for every $t\in[0,T]$ almost surely and $\Pi_t(\psi) = \rho_t(\psi) / \rho_t(\R^d)$ for every $t\in[0,T]$ and $\psi\in\mathrm{C^2_b}(\R^d)$.
Moreover, one can introduce the martingale
\begin{equation*}
	 Z_t = \exp\left\{ -\int_0^t h(X_s)\cdot\de B_s - \frac{1}{2}\int_0^t |h(X_s)|^2\de s\right\},\quad t\in[0,T],
\end{equation*}
and set $\de\mathbb{Q}=Z_T\de\mathbb{P}$. Then, it can be shown that the observation process $Y$, which drives the stochastic integral in  \eqref{eqn: mufirstzakai}, is a Brownian motion under $\mathbb{Q}$.

\begin{remark}
	These results hold with hypotheses less restrictive than Hypotheses \ref{hp: generali}. For more precise conditions and detailed proofs of the previous results, see for instance \cite[Chapters 2,3]{baincrisan} and the references therein.
\end{remark}
\subsection{Equations of nonlinear filtering in weak form}
We present now the framework for the filtering equations introduced in \cite{szpirglas} and extended more recently by \cite{heunislucic} to our setting.  The main idea is to study the measure-valued stochastic differential equations of nonlinear filtering without relying on the original filtering problem, introducing proper notions of weak solution, pathwise uniqueness and uniqueness in law. We will now present all these definitions, following the exposition in \cite{heunislucic}.
\begin{definition}\label{def: weaksolks}
	The pair $\{(\tilde{\Omega},\mathcal{\tilde{F}},\{\tilde{\mathcal{F}}_t\},\tilde{\mathbb{P}}),(\tilde{\Pi}_t,\tilde{I_t})\}$ is a weak solution to the Kushner-Stratonovich equation starting at $\pi\in\prob(\R^d)$ if:
	\begin{enumerate}
		\item[i.] $(\tilde{\Omega},\mathcal{\tilde{F}},\{\tilde{\mathcal{F}}_t\},\tilde{\mathbb{P}})$ is a complete filtered probability space.
		\item[\textit{ii}.] $\tilde{I} = \{\tilde{I}_t,t\in[0,T]\}$ is an $\R^d$-valued $\{\tilde{\mathcal{F}}_t\}$-Brownian motion on $(\tilde{\Omega},\mathcal{\tilde{F}},\tilde{\mathbb{P}})$.
		\item[\textit{iii}.] $\tilde{\Pi}=\{\tilde{\Pi}_t,t\in[0,T]\}$ is a $\prob(\R^d)$-valued continuous $\{\tilde{\mathcal{F}}_t\}$-adapted process such that
		\begin{equation*}
			\tilde{\mathbb{P}}\left(\int_0^T\sum_{i=1}^d\tilde\Pi_s(|h_i|)^2\de s<\infty\right) = 1,
		\end{equation*}
		and for every $\psi\in\mathrm{C^2_b}(\R^d)$ it holds
		\begin{equation}\label{eqn: weakks}
			\tilde\Pi_t(\psi) = \pi(\psi) + \int_0^t \tilde\Pi_s(A\psi)\de s + \int_0^t \left( \tilde\Pi_s(h \psi+ B\psi) - \tilde\Pi_s(\psi)\tilde\Pi_s(h)\right)\cdot\de \tilde I_s,
		\end{equation}
		for every $t\in[0,T]$, almost surely.
	\end{enumerate}
\end{definition}
Analogously, we can define the Zakai equation's weak solutions:
\begin{definition}\label{def: weaksolz}
	The pair $\{(\tilde{\Omega},\mathcal{\tilde{F}},\{\tilde{\mathcal{F}}_t\},\tilde{\mathbb{Q}}),(\tilde{\rho_t},\tilde{Y_t})\}$ is a weak solution to the Zakai equation starting at $\mu\in\mathcal{M}^+(\R^d)$ if:
	\begin{enumerate}
		\item[\textit i.] $(\tilde{\Omega},\mathcal{\tilde{F}},\{\tilde{\mathcal{F}}_t\},\tilde{\mathbb{Q}})$ is a complete filtered probability space.
		\item[\textit {ii}.] $\tilde{Y} = \{\tilde{Y}_t,t\in[0,T]\}$ is an $\R^d$-valued $\{\tilde{\mathcal{F}}_t\}$-Brownian motion on $(\tilde{\Omega},\mathcal{\tilde{F}},\tilde{\mathbb{Q}})$.
		\item[\textit {iii}.] $\tilde{\rho}=\{\tilde{\rho}_t,t\in[0,T]\}$ is a $\mathcal{M}^+(\R^d)$-valued continuous $\{\tilde{\mathcal{F}}_t\}$-adapted process such that
		\begin{equation*}
			\tilde{\mathbb{Q}}\left(\int_0^T\sum_{i=1}^d\tilde{\rho}_s(|h_i\psi + B_i\psi|)^2\de s<\infty\right) = 1,
		\end{equation*}
		and for every $\psi\in\mathrm{C^2_b}(\R^d)$ it holds
		\begin{equation} \label{eqn: weakzakai}
			\tilde\rho_t(\psi) = \mu(\psi) + \int_0^t \tilde\rho_s(A\psi)\de s + \int_0^t \tilde\rho_s(h \psi+ B\psi)\cdot\de \tilde Y_s,
		\end{equation}
		for every $t\in[0,T]$, almost surely.
	\end{enumerate}
\end{definition}
\begin{remark}\label{rmk: bddmass}
	A useful result, proved in \cite{heunislucic} (see Fact 3.2), is the fact that the trajectories of weak solutions to the Zakai equation have total mass that does not touch zero and with uniformly bounded moments. More precisely, if $\{(\tilde{\Omega},\mathcal{\tilde{F}},\{\tilde{\mathcal{F}}_t\},\tilde{\mathbb{Q}}),(\tilde{\rho_t},\tilde{Y_t})\}$ is a weak solution to the Zakai equation starting at $\mu\in\mathcal{M}^+(\R^d)$, then $\tilde{\rho}_t(\R^d)>0$ for every $t\in[0,T]$, $\tilde{\mathbb{Q}}$ almost surely. Moreover for every $\alpha\in(1,+\infty)$ there exists a positive constant $\gamma(\alpha,\mu)$ such that 
\begin{equation*}
	\mathbb{E}^{\tilde{\mathbb{Q}}}\left[\sup_{t\in[0,T]} |\tilde{\rho}_t(\R^d)|^\alpha\right]\leq\gamma(\alpha,\mu).
\end{equation*} 
\end{remark}
\begin{remark}\label{rmk: measinitcond}
	In \cite{szpirglas} and \cite{heunislucic}, the Zakai equation is always taken with initial condition in the space of probability measures. However, it easy to consider the case in which the initial condition is a positive measure, different from the null measure. Indeed one can always reconduct the problem to the one starting from a probability by performing a standardization, thanks to the linearity of the Zakai equation.
\end{remark}
In the following lemma we show that if a weak solution to a filtering equation starts from a measure with finite second moment, then its trajectories will take value in a space of measures with finite second moment. We postpone the proof of Lemma \ref{lemma: solinm2p2} to \hyperref[app: a]{Appendix} to avoid technicalities in this expository chapter on nonlinear filtering.
\begin{lemma}\label{lemma: solinm2p2}
	Let $\{(\tilde{\Omega},\mathcal{\tilde{F}},\{\tilde{\mathcal{F}}_t\},\tilde{\mathbb{Q}}),(\tilde{\rho_t},\tilde{Y_t})\}$ be a weak solution to the Zakai equation starting at $\mu\in\mathcal{M}_2^+(\R^d)$. Then $\tilde{\rho_t}\in\mathcal{M}_2^+(\R^d)$ for every $t\in[0,T]$, $\tilde{\mathbb{Q}}$-almost surely. Similarly, if $\{(\tilde{\Omega},\mathcal{\tilde{F}},\{\tilde{\mathcal{F}}_t\},\tilde{\mathbb{P}}),(\tilde{\Pi}_t,\tilde{I_t})\}$ is a weak solution to the Kushner-Stratonovich equation starting at $\pi\in\mathcal{P}_2(\R^d)$, then $\tilde{\Pi_t}\in\mathcal{P}_2(\R^d)$ for every $t\in[0,T]$, $\tilde{\mathbb{P}}$-almost surely.
\end{lemma}
Regarding the uniqueness, we have the following two definitions, which follow the classical Yamada-Watanabe formalism.
\begin{definition} The Kushner-Stratonovich equation has the pathwise uniqueness property if: given two weak solutions $\{(\tilde{\Omega},\mathcal{\tilde{F}},\{\tilde{\mathcal{F}}_t\},\tilde{\mathbb{P}}),(\tilde{\Pi}^1_t,\tilde{I_t})\}$ and $\{(\tilde{\Omega},\mathcal{\tilde{F}},\{\tilde{\mathcal{F}}_t\},\tilde{\mathbb{P}}),(\tilde{\Pi}^2_t,\tilde{I_t})\}$ of the equation starting at $\pi\in\prob(\R^d)$, it holds that
\begin{equation*}
	\tilde{\mathbb{P}}\left(\tilde{\Pi}_t^1 =\tilde{\Pi}_t^2\quad, \forall t\in[0,T]\right) = 1.
\end{equation*}
In the same way we state the pathwise uniqueness property for the Zakai equation. 
\end{definition}
\begin{definition} The Kushner-Stratonovich equation has the uniqueness in joint law property if: given two weak solutions $\{(\tilde{\Omega},\mathcal{\tilde{F}},\{\tilde{\mathcal{F}}_t\},\tilde{\mathbb{P}}),(\tilde{\Pi}_t,\tilde{I_t})\}$ and $\{(\hat{\Omega},\mathcal{\hat{F}},\{\hat{\mathcal{F}}_t\},\hat{\mathbb{P}}),$ $(\hat{\Pi}_t,\hat{I_t})\}$ of the equation starting at $\pi\in\prob(\R^d)$, it holds that the processes $(\tilde{\Pi},\tilde I) = \{(\tilde{\Pi}_t,\tilde I_t),t\in[0,T]\}$ and $(\hat{\Pi}, \hat I) = \{(\hat{\Pi}_t, \hat I_t),t\in[0,T]\}$ have the same finite dimensional distributions, where we endowed $\mathcal{P}(\R^d)$ with the Borel $\sigma$-algebra induced by the weak convergence topology. In the same way we define the uniqueness in joint law property for the Zakai equation.
\end{definition}
The main result in \cite{heunislucic} is the following theorem regarding the uniqueness for the two equations of nonlinear filtering:
\begin{theorem}[\cite{heunislucic}]\label{thm: uniqhl} Assume that Hypotheses \ref{hp: generali} hold. Then:
	\begin{enumerate}
		\item[i.] the Zakai equation has the pathwise uniqueness and the uniqueness in joint law properties;
		\item[ii.] the Kushner-Stratonovich equation has the uniqueness in joint law property.
	\end{enumerate}
\end{theorem}
\begin{remark}\label{rmk: strong}
In the case without correlated noise, studied in \cite{szpirglas}, it is possible to prove pathwise uniqueness also for the weak solutions of the Kushner-Stratonovich equation. Moreover, in \cite{szpirglas} (Théorème V.6)  it is shown how the classical Yamada-Watanabe result also apply to this situation, namely pathwise uniqueness and existence of a weak solution implies existence of a strong solution. 
The technique is not affected by the addition of a correlated noise, so at least for the Zakai equation we also have existence of a strong solution. 
This allows us to fix a probability space $(\tilde{\Omega},\mathcal{\tilde{F}},\{\tilde{\mathcal{F}}_t\},\tilde{\mathbb{Q}})$, equipped with a $\{\tilde{\mathcal{F}}_t\}$-Brownian motion $\tilde{Y}$, and to solve the Zakai equation with respect to different initial conditions on the same probabilistic setup.

\end{remark}
\begin{remark}
	Since the uniqueness in law property holds, the Markov property follows for both the weak solution of the Kushner-Stratonovich and the Zakai equation.
\end{remark}
\begin{remark}\label{rmk: lesshp}
	The fact that $a:=\sigma\sigma^\top$ must be uniformly elliptic (Hypotheses \ref{hp: generali} - \textit{b.}) is necessary only for the proof of Proposition \ref{prop: lderzakai} (in particular, it is required to ensure the existence of a smooth solution to \eqref{eqn: adjeq}). For all the previous results, one can just assume a non-degeneracy condition on $a$, that is $a(x)$ positive definite for every $x\in\R^d$. 
	Moreover, the non-degeneracy of $\sigma$ is required only for Theorem \ref{thm: uniqhl}. Thus, if one assumes the conclusions of Theorem \ref{thm: uniqhl} and the Markov property for the solutions, as well as the existence of a smooth solution to \eqref{eqn: adjeq}, then all the following results still hold without Hypotheses \ref{hp: generali} - \textit{b.} 
\end{remark}
We conclude this section by stating how to obtain a weak solution to the Kushner-Stratonovich equation from a weak solution to the Zakai equation and viceversa. First, let us assume that Hypotheses \ref{hp: generali} hold and let $\{(\tilde{\Omega},\mathcal{\tilde{F}},\{\tilde{\mathcal{F}}_t\},\tilde{\mathbb{P}}),(\tilde{\Pi}_t,\tilde{I_t})\}$ be a weak solution to the Kushner-Stratonovich equation. Then, we can define two processes $\tilde{Y}=\{\tilde{Y}_t,t\in[0,T]\}$ and $\chi \defeq\{\chi_t,t\in[0,T]\}$ by
\begin{equation}
	\tilde{Y}_t = \tilde{I}_t + \int_0^t\tilde{\Pi}_s(h)\de s,\quad \chi_t = \exp\left\{-\int_0^t \tilde{\Pi}_s(h)\de \tilde{I}_s + \frac{1}{2}\int_0^t|\tilde{\Pi}_s(h)|^2\de s\right\},
\end{equation}
where it easy to see that $\chi$ in a $\tilde{\mathbb {P}}$-martingale. Thus, if we introduce the probability measure $\de \tilde{\mathbb Q} = \chi_T\de \tilde{\mathbb P}$ and $\tilde {\rho} =\{\tilde{\rho}_t =  \mu(\R^d)\chi_t^{-1}\tilde{\Pi}_t,t\in[0,T]\} $, we have that $\{(\tilde{\Omega},\mathcal{\tilde{F}},\{\tilde{\mathcal{F}}_t\},\tilde{\mathbb{Q}}),(\tilde{\rho_t},\tilde{Y_t})\}$ is a weak solution to the Zakai equation starting at $\mu\in\mathcal{M}_2^+(\R^d)$. We remark that the presence of $\mu(\R^d)$ in the definition of $\tilde{\rho}$ is necessary to keep track that the initial condition $\mu$ is not a probability measure (see also Remark \ref{rmk:  measinitcond}). On the other hand, if $\{(\tilde{\Omega},\mathcal{\tilde{F}},\{\tilde{\mathcal{F}}_t\},\tilde{\mathbb{Q}}),(\tilde{\rho_t},\tilde{Y_t})\}$ is a weak solution to the Zakai equation, we can set
\begin{equation}
	\tilde{I}_t = \tilde{Y}_t - \int_0^t\frac{\tilde{\rho}_s(h)}{\tilde{\rho}_s(\R^d)}\de s,\quad \xi_t = \exp\left\{\int_0^t \frac{\tilde{\rho}_s(h)}{\tilde{\rho}_s(\R^d)}\de \tilde{Y}_s - \frac{1}{2}\int_0^t\left|\frac{\tilde{\rho}_s(h)}{\tilde{\rho}_s(\R^d)}\right|^2\de s\right\},
\end{equation}
and, since $\xi$ is a martingale, introduce $\de \tilde{\mathbb{P}} = \xi_T\de\tilde{\mathbb Q}$ and $\tilde{\Pi} = \{\tilde{\Pi}_t = \tilde{\rho}_s / \tilde{\rho}_s(\R^d),t\in[0,T]\}$. Thus, the couple $\{(\tilde{\Omega},\mathcal{\tilde{F}},\{\tilde{\mathcal{F}}_t\},\tilde{\mathbb{P}}),(\tilde{\Pi}_t,\tilde{I_t})\}$ is a weak solution to the Kushner-Stratonovich equation.

The proofs of these results in the context of weak solutions to the nonlinear filtering equations can be found in \cite{szpirglas} for the case without correlated noise. However, the case with correlated noise has no extra difficulties. The technique used to prove these relations is based on Girsanov's theorem and on the classical It\^o formula, and it follows the way to link the filter $\Pi$ and the unnormalized filter $\rho$. A discussion of this change of probability method in the context of the filtering problem can be found for instance in \cite{baincrisan} or in \cite{xiong}.

\section{It\^o formula and regularity with respect to the initial condition for the Zakai Equation}\label{sec: zakai}
In this section we discuss some properties of the solution to the Zakai equation. Our first aim is to show a chain rule of It\^o type for the composition with a function in $\mathrm{C^2_L}(\mathcal{M}_2^+(\R^d))$. Then, we investigate the differentiability of a solution with respect to the initial condition. In this section, we will consider a complete filtered probability space $(\Omega, \mathcal{F},\{\mathcal{F}_t\},\mathbb{Q})$ endowed with a $\{\mathcal{F}_t\}$-Brownian motion $Y$, and a solution to the Zakai equation $\rho$ on that probabilistic setup (see also Remark \ref{rmk: strong}). In particular
\begin{equation}\label{eqn: zakai}
	\rho_t(\psi) = \mu(\psi) + \int_0^t \rho_s(A\psi)\de s + \int_0^t \rho_s(h \psi+ B\psi)\cdot\de Y_s,\quad \psi\in \mathrm{C^2_b}(\R^d),
\end{equation}
where the coefficients are related to the filtering problem \eqref{eqn: signal}-\eqref{eqn: obs} and the operators $A$ and $B$ are defined by \eqref{eqn: opdiff}. We will also assume that Hypotheses \ref{hp: generali} hold, so from Theorem \ref{thm: uniqhl} we have the pathwise uniqueness property. Finally, we will consider only initial conditions $\mu$ in $\mathcal{M}_2^+(\R^d)$, so $\rho$ is a  $\mathcal{M}_2^+(\R^d)$-valued process.
\subsection{It\^o formula for the Zakai Equation}\label{ssec: itoforzakai}
The purpose of this section is to identify an It\^o formula for the composition of a regular function and a process that solves \eqref{eqn: zakai}. This is a key step for our final purpose, that is write and study the backward Kolmogorov equation associated to the Zakai equation.\\\\ 
Before stating the main result of this section, we need to introduce a notation for the integral with respect to the product measure $\mu\otimes\mu$, with $\mu\in\Mp$. Let $f,g\in\mathrm{C_b}(\R^d)$, $h\in\mathrm{C_b}(\R^d\times\R^d)$. We adopt the following notations:
\begin{itemize}
	\item Every time $\mu(\de x)\otimes\mu(\de y)$ integrates a product of functions over $\R^d$, it is meant that the first one is integrated with respect to $\mu(\de x)$ and the second one with respect to $\mu(\de y)$, that is
	\begin{equation*}
		\mu\otimes\mu(fg):=\int\int f(x)g(y)\mu(\de x)\mu(\de y).
	\end{equation*} 
	In particular, $\mu\otimes \mu (f \ f) = \int\int f(x)f(y)\mu(\de x)\mu(\de y)$;
	\item  Every time $\mu(\de x)\otimes\mu(\de y)$ integrates a product of functions over $\R^d$ and a function over $\R^d\times\R^d$, it is meant that the first one is integrated with respect to $\mu(\de x)$, the second one with respect to $\mu(\de y)$ and the third with respect to both, that is
	\begin{equation*}
		\mu\otimes\mu(fgh):=\int\int f(x)g(y)h(x,y)\mu(\de x)\mu(\de y).
	\end{equation*} 
\end{itemize}
The extension to vector-valued and matrix-valued functions is straightforward.

\begin{proposition}\label{prop: itoz}
	Let $\rho=\{\rho_t,t\in[0,T]\}$ be a solution to the Zakai equation starting at $\mu\in\mathcal{M}^+_2(\R^d)$ and let $u$ be in $\mathrm {C_L^2}(\mathcal{M}_2^+(\R^d))$. Let us also assume that Hypotheses \ref{hp: generali} are satisfied. Then the following It\^o formula holds:
	\begin{equation}\label{eqn: itozakai}
		\begin{aligned}
			u(\rho_t) &= u(\mu) + \int_0^t\rho_s\left(\diff _\mu u(\rho_s)\cdot f\right)\de s\\
			 &\ +\int_0^t \frac{1}{2}\rho_s\left(\tr\left\{\diff _x\diff _\mu u(\rho_s)\sigma\sigma^\top\right\}\right)\de s+ \int_0^t\frac{1}{2}\rho_t\left(\tr\left\{\diff _x\diff _\mu u(\rho_s)\bar{\sigma}\bar{\sigma}^\top\right\}\right)\de s \\
			&\ +\int_0^t\rho_s\left(h\lf u(\rho_s)\right)\cdot\de Y_s  +\int_0^t\rho_s\left(\bar{\sigma}^\top \diff _\mu u(\rho_s)\right)\cdot\de Y_s  \\
			& \ + \int_0^t\frac{1}{2}\rho_s\otimes\rho_s\left(\lf^2u(\rho_s)h\cdot h\right)\de s + \int_0^t\rho_s\otimes\rho_s\left(h\cdot\bar{\sigma}^\top\lf \diff _\mu u(\rho_s)\right)\de s\\
			& \ + \int_0^t\frac{1}{2}\rho_s\otimes\rho_s\left(\tr\left\{\diff _\mu^2u(\rho_s)\bar{\sigma}\bar{\sigma}^\top\right\}\right)\de s,\quad t\in[0,T],
		\end{aligned}
	\end{equation}
almost surely in $\Omega$.
\end{proposition}
\begin{proof}
The proof in divided into five steps: the idea is to show the formula for cylindrical function over a compact subset of $\mathcal{M}_2^+(\R^d)$, which is a direct consequence of the classical It\^o formula, and then achieve  the result by approximation and localization. In particular, for the first four steps we assume that, for a fixed $k>1$, $\rho_t\in\mathcal{M}_{2,k}^+(\R^d)$ for every $t\in[0,T]$, that is $\rho_t(\R^d)\in[\frac{1}{k},k]$ for every $t\in[0,T]$, almost surely. 
In the last step we get rid of this condition by a localization argument.\\\\
\textit{First step.} We prove the formula for $u\in\mathcal{C}_2(\Mp)$. More precisely, $u(\mu) = g(\scalprod{\mu}{\psi_1},\dots,\scalprod{\mu}{\psi_n})$, $n\in\N$, $g\in \rm C_b^2(\R^n)$, $\{\psi_i\}_{i=1}^n\subset \rm C_b^2(\R^d)$. Without loss of generality, we discuss the case $n=1$, that is $u(\mu) = g(\scalprod{\mu}{\psi})$. The result with $n\geq1$ is obtained with the same procedure. By the classical It\^o formula and \eqref{eqn: zakai}, we get that
\begin{multline*}
	\de u(\rho_t) =\de g(\scalprod{\rho_t}{\psi})  =g'(\scalprod{\rho_t}{\psi})\rho_t(A\psi)\de t + g'(\scalprod{\rho_t}{\psi})\rho_t(h\psi + B\psi)\cdot\de Y_t \\+ \frac{1}{2}g''(\scalprod{\rho_t}{\psi})\rho_t(h\psi + B\psi)\cdot\rho_t(h\psi + B\psi))\de t\\
	=g'(\scalprod{\rho_t}{\psi})\rho_t(\diff _x\psi\cdot f)\de t +  \frac{1}{2}g'(\scalprod{\rho_t}{\psi})\rho_t\left(\tr\left\{\diff _x^2\psi\sigma\sigma^\top\right\}\right)\de t \\+  \frac{1}{2}g'(\scalprod{\rho_t}{\psi})\rho_t\left(\tr\left\{\diff _x^2\psi\bar{\sigma}\bar{\sigma}^\top\right\}\right)\de t+ g'(\scalprod{\rho_t}{\psi})\rho_t(h\psi +\bar{\sigma}^\top \diff _x\psi)\cdot\de Y_t \\+ \frac{1}{2}g''(\scalprod{\rho_t}{\psi})\rho_t(h\psi + \bar{\sigma}^\top \diff _x\psi)\cdot\rho_t(h\psi + \bar{\sigma}^\top \diff _x\psi))\de t.
\end{multline*}
Then, recalling Example \ref{ex: cyl}, we get that
\begin{multline}\label{eqn: itozdiff}
	\de u (\rho_t) = \rho_t(\diff _\mu u(\rho_t)\cdot f)\de t + \frac{1}{2} \rho_t\left(\tr\left\{\diff _x\diff _\mu u (\rho_t)\sigma\sigma^\top\right\}\right)\de t+ \frac{1}{2} \rho_t\left(\tr\left\{\diff _x\diff _\mu u (\rho_t)\bar{\sigma}\bar{\sigma}^\top\right\}\right)\de t\\
	+ \rho_t(h\lf u(\rho_t))\cdot\de Y_t + \rho_t(\bar{\sigma}^\top \diff _\mu u (\rho_t))\cdot\de Y_t
	+ \frac{1}{2}\rho_t\otimes\rho_t(\lf^2u(\rho_t)h\cdot h)\de t \\+ \rho_t\otimes\rho_t(h\cdot\bar{\sigma}^\top\lf \diff _\mu u(\rho_t)) \de t+ \frac{1}{2}\rho_t\otimes\rho_t\left(\tr\left\{\diff _\mu^2 u(\rho_t)\bar{\sigma}\bar{\sigma}^\top\right\}\right)\de t,\quad
\end{multline}
for every $t\in[0,T]$, almost surely in $\Omega$. \\\\
\textit{Second step.} Let us fix $N\geq 1$ and $k>1$. Now we show the formula for functions of the form $u(\mu) = \scalprod{\frac{\mu^r}{\mu(\R^d)^r}}{\varphi(\cdot,\dots,\cdots,\mu(\R^d))}$, with $\mu\in\mathcal{H}_N^k$, $r\in\N$, $\varphi\in \mathrm {C_b^2}(K_N^r\times\left[\frac{1}{k},k\right])$ and $\varphi$ symmetrical in the first $r$ arguments. Thanks to Lemma \ref{lemma: approxcilgencil}, there exists $\{u^n\}_{n\geq 1}\subset\mathcal{C}_2(\Mp)$ such that $\norm{u - u^n}_{\mathrm{C^2_L}(\mathcal{H}_N^k)}\to 0$ as $n\to+\infty$, where the norm has been introduced in \eqref{eqn: c2lnrom}. Thus, thanks to the first step, we get the formula \eqref{eqn: itozdiff} with $u^n$ in place of $u$.
 
We study now the convergence of the terms in the expression we obtained. Since $u^n$ converges to $u$ for every $\mu\in\mathcal{H}_N^k$, we have that $u^n(\rho_t)\to u(\rho_t)$ and $u^n(\mu)\to u (\mu)$ almost surely in $\Omega$, as $n\to+\infty$. 

For the integrals in time, we can proceed by dominated convergence thanks to the convergence in $\norm{\cdot}_{\rm C_L^2,\mathcal{H}_N^k}$ norm of  $\{u^n\}_{n\geq1}$ and the boundedness of the coefficients $b,\sigma,\bar{\sigma},h$. Here we discuss the convergence of $\int_0^t\frac{1}{2}\rho_s\otimes\rho_s\left(\tr\left\{\diff _\mu^2u^n(\rho_s)\bar{\sigma}\bar{\sigma}^\top\right\}\right)\de s$, but the other terms can be studied analogously. Thanks to Lemma \ref{lemma: approxcilgencil}, we have
\begin{equation}
	\left\vert\tr\left\{D^2_\mu u^n(\rho_s)\bar{\sigma}\bar{\sigma}^\top\right\} - \tr\left\{D^2_\mu u(\rho_s)\bar{\sigma}\bar{\sigma}^\top\right\}\right\vert\leq 3\norm{\sigma}_\infty^2\norm{\diff _\mu^2 u}_\infty,
\end{equation}
so by dominated convergence 
\begin{equation*}
	\left\vert \rho_s\otimes\rho_s\left(\tr\left\{D^2_\mu u^n(\rho_s)\bar{\sigma}\bar{\sigma}^\top\right\} - \tr\left\{D^2_\mu u(\rho_s)\bar{\sigma}\bar{\sigma}^\top\right\}\right)\right\vert\to0,
\end{equation*}
as $n\to+\infty$. Moreover, we have that for every $t\in[0,T]$
\begin{equation*}
	\left\vert \rho_s\otimes\rho_s\left(\tr\left\{D^2_\mu u^n(\rho_s)\bar{\sigma}\bar{\sigma}^\top\right\} - \tr\left\{D^2_\mu u(\rho_s)\bar{\sigma}\bar{\sigma}^\top\right\}\right)\right\vert\leq3k\norm{\sigma}_\infty^2\norm{\diff _\mu^2 u}_\infty\in L^\infty([0,t]),
\end{equation*}
since $\rho_t(\R^d)\in\left[\frac{1}{k},k\right]$. Thus, again by dominated convergence we can conclude that for any $t\in[0,T]$
\begin{equation*}
\left\vert\int_0^t\frac{1}{2}\tr\left\{D^2_\mu u^n(\rho_s)\bar{\sigma}\bar{\sigma}^\top\right\}\de s - \frac{1}{2}\int_0^t\tr\left\{D^2_\mu u(\rho_s)\bar{\sigma}\bar{\sigma}^\top\right\}\de s\right\vert\to 0,
\end{equation*}
almost surely in $\Omega$, as $n\to+\infty$.

For the stochastic integrals, we prove the convergence in $L^2(\Omega)$. Since the technique is the same for both the terms, let us focus on $\int_0^t\rho_s\left(h\lf u^n(\rho_s)\right)\cdot\de Y_s$. By It\^o isometry, we have that for any $t\in[0,T]$ it holds
\begin{multline*}
\E{\left(\int_0^t \left\{\rho_s\left(h\lf u^n(\rho_s)\right) -\rho_s\left(h \lf u(\rho_s)\right)\right\}\cdot\de Y_s\right)^2} \\= \E{\int_0^t \left\vert\left\{\rho_s\left(h\lf u^n(\rho_s)\right) -\rho_s\left(h \lf u(\rho_s)\right)\right\}\right\vert^2\de s} \to 0,
\end{multline*}
where the convergence is obtained thanks to the uniform convergence over $\mathcal{H}_N^k\times K_N$ of $\{\lf u^n\}_{n\geq1}$ and the boundedness of $h$, combined with the dominated convergence argument we used for the deterministic integral.\\\\
Every convergence we proved implies the convergence in probability, so the relation \eqref{eqn: itozakai} holds almost everywhere in $\Omega$, for every $t\in[0,T]$. Since both the right hand side and the left hand side of \eqref{eqn: itozakai} are continuous 
, the relation holds for every $t\in[0,T]$ almost everywhere in $\Omega$.
\\\\\textit{Third step.} Let $u$ be in $\mathrm {C_L^2}(\mathcal{H}_N^k)$. Then by Lemma \ref{lemma: approxcilgen}, there exists a sequence $\{\phi^n\}_{m\geq1}$ which converges pointwise to $u$, and the same holds for the derivatives needed in the It\^o formula. Moreover, as we pointed out in Remark \ref{rmk: cilgen}, $\phi^n(\mu) = \scalprod{\frac{\mu^{\times n}}{\mu(\R^d)^n}}{\varphi_n(\cdot,\dots,\cdot,\mu(\R^d))}$, with $\varphi_n\in \mathrm {C_b^2}(K_N^n\times\left[\frac{1}{k},k\right]),n\geq1$. Thus, by step two, \eqref{eqn: itozakai} holds for every $\phi^n$. To conclude, we can pass to the limit with the same argument we used in step two, exploiting the bounds on the norms given by Lemma \ref{lemma: approxcilgen}, the boundedness of $b,\sigma,\bar{\sigma},h$, the fact that $\rho_t(\R^d)\in\left[\frac{1}{k},k\right]$ for every $t\in[0,T]$ and the dominated convergence theorem.
\\\\\textit{Fourth step.} Let $u$ be in $\rm C_L^2(\mathcal{M}_{2,k}^+(\R^d))$. Thanks to Lemma \ref{lemma: approxsuppcomp}, we can conclude that \eqref{eqn: itozakai} holds also for this class of functions, by the same argument we use in the previous steps.
\\\\\textit{Fifth step.} Let us introduce the sequence of random times $\tau_k\colon\Omega\to [0,+\infty]$,
\begin{equation*}
	\tau_k = \left\{t\geq0: \rho_t(\R^d)\in \left[1/k,k\right]^\mathsf{c}\right\},\quad k >1.
\end{equation*}
First, $\{\tau_k\}_{k>1}$ are stopping times since they are exit time from a Borel set and moreover, thanks to Remark \ref{rmk: bddmass}, $\tau_k\to+\infty$, $\mathbb{P}$ almost surely as $k\to+\infty$. Then, we can consider the stopped process $\rho_t^k \defeq\rho_{t\wedge\tau_k}$ for which, thanks to te previous steps, \eqref{eqn: itozakai} holds. Indeed it still satisfies the Zakai equation \eqref{eqn: zakai} and $\rho_t^k(\R^d)\in[1/k,k]$, for every $t\in[0,T]$ and $k>1$. To conclude, we can let $k\to+\infty$ in the It\^o formula for $\rho_t^k$, recovering \eqref{eqn: itozakai} for $\rho_t$ and $u\in \mathrm {C^2_L}(\mathcal{M}_{2}^+(\R^d))$, thanks to the continuity in time of all the terms involved in the equation.

\end{proof}
\begin{remark}
	We can rewrite the formula \eqref{eqn: itozakai} in the following way
	\begin{equation*}
		\de u(\rho_t) = \rho_t\left( A\lf u (\rho_t)\right)\de t + \rho_t\left( (h + B)\lf u (\rho_t)\right)\cdot\de Y_t + \frac{1}{2}\rho_t\otimes\rho_t\left((h + B)\cdot(h + B)\lf^2 u (\rho_t)\right)\de t,
	\end{equation*}
	where
	\begin{multline*}
		\rho_t\otimes\rho_t\left((h + B)\cdot(h + B)\lf^2u(\rho_t)\right) \\= \int \! \! \! \int (h(x) + \bar{\sigma}^\top(x)\diff _x)\cdot(h(y) + \bar{\sigma}^\top(y)\diff _y)\lf^2u(\rho_t,x,y)\rho_t(\de x)\rho_t(\de y).
	\end{multline*}
\end{remark}
\begin{corollary}\label{cor: itoztime}
	Assume that Hypotheses \ref{hp: generali} hold, let $\rho=\{\rho_t,t\in[0,T]\}$ be a solution to the Zakai equation and let $u\colon[0,T]\times\M\to\R$ be in $\mathrm {C^2_L}(\mathcal{M}_s^+(\R^d))$ for the measure argument, in $\mathrm {C^1}([0,T])$ for the time argument and let $u$ and all its derivatives be bounded in all their arguments. Then it holds
	\begin{multline*}
		u(t,\rho_t) = u(0,\mu) \\+ \int_0^t\partial_su(s,\rho_s)\de s +  \int_0^t \rho_s\left( A\lf u (s,\rho_s)\right)\de s  + \int_0^t \rho_s\left( (h + B)\lf u (s,\rho_s)\right)\cdot\de Y_s\\
		+ \frac{1}{2}\int_0^t\rho_s\otimes\rho_s\left((h + B)\cdot(h + B)\lf^2 u (s,\rho_s)\right)\de s,\quad t\in[0,T],
	\end{multline*}
	almost surely.
\end{corollary}
\begin{proof}
	The proof is basically the same of Proposition \ref{prop: itoz}, with a standard modification in order to deal with the time dependence.
\end{proof}

\subsection{Differentiability properties}\label{ssec: diffmeas}
Given a complete filtered probability space $(\Omega, \mathcal{F},\{\mathcal{F}_t\},\mathbb{Q})$ endowed with a $\{\mathcal{F}_t\}$-Brownian motion $Y$, let $\rho^{s,\mu}$ be a solution to the Zakai equation. We use the superscript ${s,\mu}$ to highlight the initial value $\mu\in\mathcal{M}^+_2(\R^d)$ and the initial time $s\in[0,T]$. The aim of this subsection is to investigate its differentiability with respect to the initial condition.

By computing formally the linear functional derivative of the equation \eqref{eqn: zakai},  
for every $x\in\R^d$ we get the following equation, defined over $(\Omega, \mathcal{F},\{\mathcal{F}_t\},\mathbb{Q},Y)$, for an $\M$-valued process $Z^s(x) = \{Z_t^s(x),t\in[s,T]\}$:
\begin{equation}\label{eqn: zakaider}
	\begin{aligned}
	\scalprod{Z^s_t(x)}{\psi}=\scalprod{\delta_x}{\psi} + \int_s^t \scalprod{Z^s_\tau(x)}{A\psi} \de\tau 
	+\int_s^t \scalprod{Z^s_\tau(x)}{B\psi + h\psi}\cdot\de Y_\tau.
	\end{aligned}
\end{equation}
We can look for solutions to \eqref{eqn: zakaider} that are in $\mathcal{M}_2^+(\R^d)$ for every fixed $x\in\R^d$ and since it is a Zakai equation with initial condition in $\ptwo$ and with the same coefficients of \eqref{eqn: zakai}, we have that there exists a unique $\mathcal{M}_2^+(\R^d)$-valued solutions $Z_t(x)$ for every $x$.  
\begin{remark}
	The solution of \eqref{eqn: zakaider} will play the role of linear functional derivative of the mapping $\mathcal{M}_2^+(\R^d)\ni\mu\mapsto\rho^{s,\mu}_t$, $t\in[s,T]$, see also Remark \ref{rmk: abuseofnot}. We can notice that $Z^s(x)$ does not depend on $\mu$, as expected since the Zakai equation is linear.
\end{remark}

Before presenting the main results regarding the properties of the process $Z^s(x)$ introduced above, we provide an explicit estimate for the mass of a solution of the Zakai equation. We report the proof for completeness, even if the result is well known (see for instance Fact 3.2 in \cite{heunislucic}).
	\begin{lemma}\label{lemma: est1}
	Let $\rho^{s,\mu}$ be a  solution of the Zakai equation \eqref{eqn: zakai} and let Hypotheses \ref{hp: generali} be satisfied. Then it holds
	\begin{equation*}
		\E{\left\vert\scalprod{\rho^{s,\mu}_t}{\mathbf{1}}\right\vert^2}\leq2\scalprod{\mu}{\mathbf{1}}^2e^{2T\norm{h}^2_\infty},
	\end{equation*}
	where $\mathbf{1}$ is the function equal to $1$ for every $x\in\R^d$ and $\mu\in\mathcal{M}_2^+(\R^d)$. 
\end{lemma}  
\begin{proof}
	Since $\rho^\mu$ solves \eqref{eqn: zakai}, we can write the Zakai equation for $\psi \equiv \mathbf{1}$
	and then take the expected value. Thus, for every $t\in[0,T]$, we get
	\begin{equation*}
		\E{\left\vert\scalprod{\rho^{s,\mu}_t}{\mathbf{1}}\right\vert^2}\leq 2\scalprod{\mu}{\mathbf{1}} ^2+ 2\norm{h}_\infty^2\int_s^t\E{\left\vert\scalprod{\rho^{s,\mu}_\tau}{\mathbf{1}}\right\vert^2}\de \tau,
	\end{equation*}
	and the thesis follows by Gronwall's lemma. 
\end{proof}

\begin{proposition}\label{prop: lfzakai1}
	Let $s\in[0,T)$. Let $\rho^{s,\mu}$ be the solution of \eqref{eqn: zakai} starting at $\mu\in\mathcal{M}_2^+(\R^d)$ and let $Z^s_t(x)$, $x\in\R^d$, be the solution of \eqref{eqn: zakaider}. Then
	\begin{enumerate}
		\item[i.]  for every $m,m'\in \mathcal{M}_2^+(\R^d)$ it holds
		\begin{equation*}
			\scalprod{\rho_t^{s,m'}}{\psi} -\scalprod{\rho_t^{s,m}} {\psi} =\int_0^1\int_{\R^d}\scalprod{Z^s_t(x)}{\psi}\left[m'-m\right](\de x)\de\theta,\quad\forall t\in [s,T],
		\end{equation*} almost surely;
		\item[ii.] for every $t\in[s,T]$ and $x\in\R^d$, there exists a constant $C = C(T,h)>0$ such that
		\begin{equation*}
			\E{|\scalprod{Z^s_t(x)}{\mathbf{1}}|^2}\leq C(T,h).
		\end{equation*}
	\end{enumerate}
\end{proposition}

\begin{proof}
	In this proof we hide the dependence on $s$ in $\rho_t^{s,\mu}$ and $Z^s_t(x)$. Let us define, for every $t\in[s,T]$ and every $m,m'\in\mathcal{M}^+_2(\R^d)$, the mapping 
	\begin{equation*}
		\tilde{Z}^{m,m'}_t\colon\mathrm{C^2_b}(\R^d)\ni\psi\mapsto\scalprod{\tilde{Z}^{m,m'}_t}{\psi}\defeq\int_{\R^d}\scalprod{Z_t(x)}{\psi}[m'-m](\de x).
	\end{equation*}
	It is easy to check that $\tilde{Z}^{m,m'}_t\in\M$ and then we can define the measure-valued process $\Delta^{m,m'} = \{\Delta^{m,m'}_t\defeq \rho_t^{m'}-\rho_t^{m}-\tilde{Z}^{m,m'}_t,t\in[s,T]\}$. Recalling that  $\rho^{m'}$ and $\rho^{m}$ solve \eqref{eqn: zakai} and $Z(x)$ solves \eqref{eqn: zakaider}, by linearity we obtain that, for every $m,m'\in\mathcal{M}_2^+(\R^d)$, $\Delta^{m,m'}$ solves a Zakai equation with null initial condition and same coefficients as \eqref{eqn: zakai}. Then, it holds that $\Delta^{m,m'}$ is the process equal to the null measure for every $t\in[s,T]$.
	Thus, since $Z_t(x)$ does not depend on $\mu$, we can say that for every $t\in[s,T]$
	\begin{equation*}
			\scalprod{\rho_t^{m'}}{\psi} -\scalprod{\rho_t^m} {\psi} =\int_0^1\int_{\R^d}\scalprod{Z_t(x)}{\psi}\left[m'-m\right](\de x)\de\theta,
	\end{equation*}almost surely. 
	Regarding \textit{ii}, it follows directly from Lemma \ref{lemma: est1}
\end{proof}

We also need to study the differentiability of the mapping $\R^d\ni x\mapsto\E{\scalprod{Z^{s}_t(x)}{\psi}}\in\R$, for every fixed $\psi\in \mathrm {C_b^2}(\R^d)$ and $t\in[s,T]$. 
\begin{proposition}\label{prop: lderzakai}
	Let $Z^{s}(x)$ be the solution of the equation \eqref{eqn: zakaider} and let Hypotheses \ref{hp: generali} hold. Then, for every $\psi\in \mathrm {C_b^2}(\R^d)$ and $t\in[s,T]$, the mapping $\R^d\ni x\mapsto\E{\scalprod{Z^{s}_t(x)}{\psi}}\in\R$  is twice continuously differentiable, with bounded derivatives.
\end{proposition}
Before proving Proposition \ref{prop: lderzakai}, let us introduce some auxiliary tools. Let us denote by $I^s(x)$ the intensity measure associated to $Z^s(x)$, that is the measure such that $\E{\scalprod{Z^s_t(x)}{\psi}} = \scalprod{I^s_t(x)}{\psi}$. From \eqref{eqn: zakaider} we have that
\begin{equation}\label{eqn: meanzakai}
	\E{\scalprod{Z^s_t(x)}{\psi}} = \scalprod{\delta_x}{\psi} + \int_s^t \E{\scalprod{Z^s_\tau(x)}{A\psi}}\de \tau,\quad t\in[s,T],\quad \psi\in\mathrm{C^2_b(\R^d)},
\end{equation}
and so $I^s(x)$ solves the following Fokker-Planck equation:
\begin{equation}\label{eqn: fokkerplanck}
	{\scalprod{I^s_t(x)}{\psi}} = \scalprod{\delta_x}{\psi} + \int_s^t \scalprod{I^s_\tau(x)}{A\psi}\de \tau,\quad t\in[s,T],\quad \psi\in\mathrm{C^2_b(\R^d)}.
\end{equation}
Following the argument used for instance to prove Proposition 6.1.2 in \cite{bogachevkrylovrocknershaposhnikov} or Lemma 4.8 in  \cite{baincrisan}, from \eqref{eqn: fokkerplanck} one can deduce that for every $\varphi\in\mathrm{C^{1,2}_b}(\R^d\times[s,t])$ it holds that
\begin{equation}\label{eqn: fokkerplanckwithtime}
	{\scalprod{I^s_t(x)}{\varphi_t}} = \scalprod{\delta_x}{\varphi_s} + \int_s^t \scalprod{I^s_\tau(x)}{(\partial_\tau + A)\varphi_\tau}\de \tau.
\end{equation}
Starting from \eqref{eqn: fokkerplanckwithtime}, we can prove Proposition \ref{prop: lderzakai}. The argument exploits the regularity of the solution of a suitable auxiliary backward partial differential equation.
\begin{proof}[Proof of Proposition \ref{prop: lderzakai}]
Let us fix $t\in [s,T]$, $\psi\in\mathrm{C^2_b}(\R^d)$ and let us introduce the backward equation
\begin{equation}\label{eqn: adjeq}
\begin{cases}
	\partial_\tau v(y,\tau) + A v (y,\tau) = 0, &\quad (y,\tau)\in \R^d\times[s,t],\\
	v(y,t) = \psi(y), &\quad y\in\R^d.
\end{cases}
\end{equation}
Thanks to Hypotheses \ref{hp: generali} (see for instance Theorem 4.6 in \cite{friedman1} and more precisely the discussion after Theorem 5.1, page 147), we have that there exists a unique classical solution $v\in\mathrm{C^{2,1}_b}(\R^d\times[s,t])$ to \eqref{eqn: adjeq}. Then, if we choose $v$ as a test function in \eqref{eqn: fokkerplanckwithtime}, we obtain
\begin{equation}
	{\scalprod{I^s_t(x)}{v_t}} = \scalprod{\delta_x}{v_s} + \int_s^t \scalprod{I^s_\tau(x)}{(\partial_\tau + A)v_\tau}\de \tau = v(x,s),
\end{equation}
where the last equality follows from the fact that $v$ solves \eqref{eqn: adjeq}. Thus, recalling that $v(x,t) = \psi(x)$ and  the definition of $I^s(x)$, we obtain that
\begin{equation*}
	\E{\scalprod{Z^s_t(x)}{\psi}} = \scalprod{I^s_t(x)}{\psi} = v(x,s),\quad x\in\R^d,
\end{equation*}
and so the mapping $\R^d\ni x\mapsto\E{\scalprod{Z^{s}_t(x)}{\psi}}$ is in $\mathrm{C^2_b}(\R^d)$ thanks to the regularity of $v$.
\end{proof}

\begin{remark}\label{rmk: abuseofnot}
In Proposition \ref{prop: lfzakai1}, we showed that for $t\in[s,T]$ and $\psi\in\mathrm{C^2_b}(\R^d)$ fixed, the mapping $\mu\mapsto\scalprod{\rho^{s,\mu}_t}{\psi}\in L^1(\Omega)$ is in $\mathrm{C}^1(\mathcal{M}_2^+(\R^d);L^1(\Omega))$, with derivative given by $\scalprod{Z^s_t(x)}{\psi}$ and independent of $\mu$. Note that this does not imply that, almost surely, $Z^s_t$ is the linear functional derivative of $\mu\mapsto\rho_t^{s,\mu}$, since the continuity of $x\mapsto \scalprod{Z_t^s(x)}{\psi}$ holds only under expectation. 

With the same procedure used for Proposition \ref{prop: lfzakai1} we can find a process, that we denote with $U^s(x,y)$, $x,y\in\R^d$, which is symmetrical with respect to $x$ and $y$ and which satisfies properties analogue to \textit{i, ii,} in Proposition \ref{prop: lfzakai1} where $\rho^{s,\mu}$ is substituted with $Z^s (x)$ and $Z^s (x)$ with $U^s(x,y)$. Moreover, it turns out that $U^s(x,y)$ coincides with the null measure for every $t\in[s,T]$ and $x,y\in\R^d$.
\end{remark}
To conclude, we summarize in a proposition all the properties we showed in this section and which will be useful in the following discussions.
\begin{proposition}\label{prop: summarize}
	Let $s\in[0,T)$, let $(\Omega, \mathcal{F},\{\mathcal{F}_t\},\mathbb{Q}, \{Y_t\})$ be fixed and let Hypotheses \ref{hp: generali} holds. If $\rho^{s,\mu}$ is the solution of the Zakai equation \eqref{eqn: zakai}, then, for every $\psi\in\mathrm{C^2_b}(\R^d)$ and $t\in[s,T]$, the mapping  $\mu\mapsto\scalprod{\rho_t^{s,\mu}}{\psi}$ is in $\mathrm{C^2}(\mathcal{M}_2^+(\R^d);L^1(\Omega))$. Moreover, the mapping $\mu\mapsto\E{\scalprod{\rho_t^{s,\mu}}{\psi}}$ is in $\mathrm{C^2_L}(\mathcal{M}_2^+(\R^d))$.
\end{proposition}

\section{The backward Kolmogorov equation associated to the Zakai equation}\label{sec: zakaikolmogorov}
In this section we write and study the backward Kolmogorov equation associated to the Zakai equation, that is a parabolic partial differential equation on a space of positive measures. Let us denote with $\mathcal{L}$ the infinitesimal generator of the Zakai process, namely the operator $\mathcal{L}\colon\mathrm{C^2_L}(\mathcal{M}_2^+(\R^d))\to \mathrm{C_b}(\mathcal{M}_2^+(\R^d))$ defined by
\begin{equation*}
		\mathcal{L} u(\mu) = \mu\left(A\lf u(\mu)\right) + \frac{1}{2}\mu\otimes\mu\left((h + B)^\top(h + B)\lf^2 u(\mu)\right), \\ 
\end{equation*}
where $A$ and $B$ are defined by \eqref{eqn: opdiff} and $h\colon\R^d\to\R^d$ is Borel measurable and bounded.
The backward Kolmogorov equation we want to study is:
\begin{equation}\label{eqn: kolmogorov}
	\begin{cases}
		\partial_su(\mu,s) + \mathcal{L}u(\mu,s) = 0\quad&(\mu,s)\in\mathcal{M}_2^+(\R^d)\times[0,T],\\
		u(\mu,T) = \Phi(\mu)\quad&\mu\in\mathcal{M}_2^+(\R^d),
	\end{cases}
\end{equation}
where $\Phi$ is in $\mathrm {C^2_L}(\mathcal{M}_2^+(\R^d))$. Our aim is to study existence and uniqueness of classical solutions, in the sense given by the following definition:
\begin{definition}
	We say that $u\colon\mathcal{M}_2^+(\R^d)\times[0,T]\to\R$ is a classical solution to the backward Kolmogorov equation associated to the Zakai equation if it is of class $\mathrm {C^2_L}(\mathcal{M}_2^+(\R^d))$ in the measure argument and $\mathrm {C}^1([0,T])$ in the time argument (where in $t=0$ and $t=T$ the derivatives are understood in unilateral sense), if it and all its derivatives are bounded in all their arguments and if it satisfies the backward equation \eqref{eqn: kolmogorov}. 
\end{definition}

\subsection{Existence and uniqueness of a classical solution}\label{sec: exuniqzakai}
In order to show existence and uniqueness, we follow the classical approach to these kind of problems. First, we assume that a solution exists and we prove a representation formula which guarantees the uniqueness. In the following, when we refer to a solution $\rho^{s,\mu}$ to the Zakai equation, it is understood as the solution defined over a fixed complete filtered probability space $(\Omega, \mathcal{F},\{\mathcal{F}_t\},\mathbb{Q})$ endowed with a $\{\mathcal{F}_t\}$-Brownian motion $Y$ (see Remark \ref{rmk: strong}), which solves the equation starting from $\mu\in\mathcal{M}_2^+(\R^d)$ at $s\in[0,T)$.
\begin{proposition}
	Let $u=u(\mu,s)$ be a classical solution to \eqref{eqn: kolmogorov} and let $\rho^{s,\mu}$ be a solution to the Zakai equation \eqref{eqn: zakai}. Then the following representation formula holds
	\begin{equation}\label{eqn: reprformula}
		u(\mu,s) = \EQ{\Phi\left(\rho_T^{s,\mu}\right)},\quad (\mu,s)\in\mathcal{M}_2^+(\R^d)\times[0,T],
	\end{equation}
	and so the solution $u$ is uniquely characterized.
\end{proposition}
\begin{proof}
	Let us consider the composition $u(\rho^{s,\mu}_T,T)$, where $\rho_T^{s,\mu}$ is the solution to \eqref{eqn: zakai} starting at time $s$ with value $\mu$. Then, by the It\^o formula we get
	\begin{equation*}
		u(\rho^{s,\mu}_T,T) - u(\rho^{s,\mu}_s,s)  = \int_s^T\{\partial_su(\rho^{s,\mu}_\tau,\tau) + \mathcal{L} u(\rho^{s,\mu}_\tau,\tau)\}\de \tau + \int_s^T \mathcal {G}  u(\rho^{s,\mu}_\tau,\tau)\cdot\de Y_\tau,
	\end{equation*}
	where $\mathcal {G}  u (\mu,t) = \mu(h\lf u(\mu,t) + \bar{\sigma}^\top \diff _\mu u(\mu,t))$.
	First, we notice that since $ u$ is a solution of \eqref{eqn: kolmogorov}, the time integral is zero, and by taking the expectation we get
	\begin{equation*}
		\EQ{\Phi\left(\rho^{s,\mu}_T\right)} -  u(\mu,s)  = \EQ{\int_s^T \mathcal{G}  u(\rho^{s,\mu}_\tau,\tau)\cdot\de Y_\tau}.
	\end{equation*}
	The right hand side is equal to zero since the stochastic integral in the expected value is a martingale. Then the thesis follows, since for every $(\mu,s)\in\mathcal{M}_2^+(\R^d)\times [0,T]$,
	\begin{equation*}
		\EQ{\Phi\left(\rho^{s,\mu}_T\right)} =  u(\mu,s).
	\end{equation*}
\end{proof}
In order to prove the existence of a solution, we will show that a function $u$ defined by \eqref{eqn: reprformula} is regular enough and satisfies \eqref{eqn: kolmogorov}. In order to do that, we need some auxiliary results on the differentiability of $u$ with respect to the measure argument that we collect in the following proposition:
\begin{proposition}\label{prop: uinc2}
	Let $u(\mu,s) = \EQ{\Phi(\rho_T^{s,\mu})}$, where $\rho_T^{s,\mu}$ is the solution to the Zakai equation starting at time $s$ from $\mu\in\mathcal{M}_2^+(\R^d)$ and $\Phi\in \mathrm{C^2_L}(\mathcal{M}^+_2(\R^d))$. Then, for every $s\in[0,T]$, $u$ is in $\mathrm {C^2_L}(\mathcal{M}_2^+(\R^d))$.
\end{proposition}
\begin{proof}
	Let us first deal with the differentiability in linear functional sense. Proceeding as in Proposition \ref{prop: complflfm}, by a simple chain rule argument we conclude that $u(\mu,s)$ is in $\rm C^2(\mathcal{M}_2^+(\R^d))$, thanks to the fact $\Phi\in \mathrm {C_L^2}(\mathcal{M}_2^+(\R^d))$ combined with Proposition \ref{prop: summarize}. In particular we get the following formulas:
	\begin{align*}
		\lf u(\mu,x,s) &=\EQ{\scalprod{Z^s_T(x)}{\lf \Phi(\rho_T^{s,\mu})}},\\
		\lf^2 u(\mu,x,y,s) &=\EQ{\scalprod{Z^s_T(x)\otimes Z^s_t(y)}{\lf^2 \Phi(\rho_T^{s,\mu})}},
	\end{align*}
	where $\lf \Phi(\rho_T^{s,\mu},x)$ means $\lf \Phi (\mu,x)$ evaluated in $\rho_T^{s,\mu}$, and $Z^s$ is the process introduced in \eqref{eqn: zakaider}.

Regarding the differentiability of the first-order linear functional derivative with respect the additional space variable, we have that the mapping $x\mapsto\lf u(\mu,x,s)$ is twice continuously differentiable with bounded derivatives thanks to Proposition \ref{prop: summarize} and to the fact that $\lf \Phi(\rho_T^{s,\mu})\in \mathrm {C_b^2}(\R^d)$ for $\mu$ fixed.
In a similar way we can show that the mapping $(x,y)\mapsto \lf^2 u(\mu,x,y)$ is twice continuously differentiable with bounded derivatives. Indeed, if we take a symmetrical function $\psi\in\mathrm{C^2_b}(\R^d\times\R^d)$, we can obtain a version of Proposition \ref{prop: summarize} for the mapping $(x,y)\mapsto \EQ{\scalprod{Z^s_T(x)\otimes Z^s_t(y)}{\psi}}$, by combining the technique used in the proof of Proposition \ref{prop: summarize} and the ideas in the proof of Theorem 4.26 in \cite{baincrisan}.
	\end{proof}
Now that the object $\mathcal{L}u$ is well defined, we need to investigate its regularity with respect to the time.
\begin{lemma}\label{lemma: contofl}
	Let $u$ be defined by \eqref{eqn: reprformula}. Then, for every $\mu\in\mathcal{M}_2^+(\R^d)$, the mappings $[0,T]\ni s\mapsto\mathcal{L}u(\mu,s)$ and $[s,T]\times[0,T]\ni(\tau,\sigma)\mapsto\mathcal{L}u(\rho_\tau^{s,\mu},\sigma)\in L^2(\Omega)$ are continuous.
\end{lemma}
\begin{proof}
	Let us fix $t\in[s,T]$. Then, by classical estimates on \eqref{eqn: zakai}, it follows that for every $\psi\in\mathrm{C^2_b}(\R^d)$ the mapping $[0,T]\ni s\mapsto \scalprod{\rho_t^{s,\mu}}{\psi}\in L^2(\Omega)$ is continuous. Thus, combining this with the expression for the derivatives of $u$ in Proposition \ref{prop: uinc2} and the boundedness of $\Phi$ with its derivatives, we get that $[0,T]\ni s\mapsto\mathcal{L}u(\mu,s)$ is continuous. Regarding $[s,T]\times[0,T]\ni(\tau,\sigma)\mapsto\mathcal{L}u(\rho_\tau^{s,\mu},\sigma)$, again we can conclude recalling that $\Phi\in \mathrm{C^2_L}(\mathcal{M}_2^+(\R^d))$ combined with Proposition \ref{prop: uinc2} and the continuity of $\rho^{s,\mu}$.
\end{proof}
Finally we can show the main result, that is the existence of a solution to \eqref{eqn: kolmogorov} via representation formula:
\begin{theorem}\label{thm: exbkwzakai}
 Let $u(\mu,s) = \EQ{\Phi(\rho_T^{s,\mu})}$, where $\rho_T^{s,\mu}$ is the solution to the Zakai equation starting at time $s$ from $\mu\in\mathcal{M}_2^+(\R^d)$, $\Phi\in\mathrm{C^2_L}(\mathcal{M}_2^+(\R^d))$ and let Hypotheses \ref{hp: generali} hold. Then it is the unique classical solution to the backward Kolmogorov equation \eqref{eqn: kolmogorov}.
\end{theorem}
\begin{proof}
	Let us fix $h$ small and positive. We want to show that
	\begin{equation}\label{eqn: 9}
		\lim_{h\to0}\frac{1}{h}\left[u(\mu,s+h)-u(\mu,s)\right] = -\mathcal{L}u(\mu,s).
	\end{equation}
	If this is true, the mapping $g\colon s\mapsto g(s)\defeq u(\mu,s)$ has right derivative in $[0,T)$. Moreover, by Lemma \ref{lemma: contofl} the right-hand term in \eqref{eqn: 9} is continuous, so $g\in\mathrm {C^1}([0,T))$ and by a standard argument it can be shown that it is continuously differentiable in $[0,T]$.\\	
	Let us show \eqref{eqn: 9}. First, thanks to the Markov property of the process $\rho^{s,\mu}$ it holds that
		$u(\mu,s) = \EQ{u\left(\rho_{s+h}^{s,\mu},s+h\right)}$.
	Then, we can proceed by applying It\^o formula and taking the expectation:
	\begin{equation*}
		u(\mu,s+h) - u(\mu,s) = \EQ{u(\mu,s+h)-u\left(\rho_{s+h}^{s,\mu},s+h\right)} = -\EQ{\int_s^{s+h}\mathcal{L} u (\rho_\tau^{s,\mu},s+h)\de\tau}.
	\end{equation*}
	To conclude, it remains to show that
	\begin{equation*}\label{eqn: 10}
		\lim_{h\to0}\frac{1}{h}\EQ{\int_s^{s+h}\mathcal{L} u (\rho_\tau^{s,\mu},s+h)\de\tau} =\mathcal{L} u(\mu,s),
	\end{equation*}
	but this follows from Lemma \ref{lemma: contofl} and mean-value theorem.
\end{proof}

\begin{remark}
	All the previous results can be extended to the time inhomogeneous case, that is when the coefficients $b,\sigma,\bar\sigma$ depend also on time, by assuming that Hypotheses \ref{hp: generali} hold with uniform in time constants.
\end{remark}
\section{The backward Kolmogorov equation associated to the Kushner-Stratonovich equation}\label{sec: kseqn}
Our goal in this last section is to prove existence and uniqueness for the backward Kolmogorov equation associated to the Kushner-Stratonovich equation. We will proceed by exploiting the relation with the Zakai equation, pointed out at the end of Section \ref{sec: eqnfilter}. Let us fix $(\Omega, \mathcal{F},\{\mathcal{F}_t\},\mathbb{Q}, \{Y_t\})$ and let $\rho^{s,\pi}$ be a solution to \eqref{eqn: zakai} starting at $\pi\in\ptwo$. Let us define the couple
\begin{equation}\label{eqn: ksfromz1}
	{I}^\pi_t = {Y}_t - \int_0^t\frac{{\rho_\tau^{s,\pi}}(h)}{{\rho_\tau^{s,\pi}}(\R^d)}\de \tau,\quad \xi^\pi_t = \exp\left\{\int_0^t \frac{{\rho_\tau^{s,\pi}}(h)}{{\rho_\tau^{s,\pi}}(\R^d)}\de {Y}_\tau - \frac{1}{2}\int_0^t\left|\frac{{\rho}_\tau^{s,\pi}(h)}{{\rho}_\tau^{s,\pi}(\R^d)}\right|^2\de \tau\right\},
\end{equation}
and set 
	\begin{equation}\label{eqn: ksfromz2}
		\de {\mathbb{P}^\pi} = \xi^\pi_T\de{\mathbb Q},\quad \Pi^{s,\pi}  = \rho^{s,\pi} / \rho^{s,\pi}(\R^d). 
	\end{equation}
As remarked in Section \ref{sec: eqnfilter}, if we assume Hypotheses \ref{hp: generali}, then the couple $\{({\Omega},\mathcal{{F}},\{{\mathcal{F}}_t\},{\mathbb{P}^\pi}),$ $({\Pi}^{s,\pi}_t,{I^\pi_t})\}$ is the unique in law weak solution to the Kushner-Stratonovich equation starting at $\pi\in\ptwo$. In particular, for every $t\in[s,T]$ and $\psi\in\mathrm{C^2_b}(\R^d)$ it holds
\begin{equation}\label{eqn: ks}
	\Pi^{s,\pi}_t(\psi) = \pi(\psi) + \int_0^t \Pi^{s,\pi}_s(A\psi)\de s + \int_0^t \left( \Pi^{s,\pi}_s(h \psi+ B\psi) - \Pi^{s,\pi}_s(\psi)\Pi^{s,\pi}_s(h)\right)\cdot\de I^{s,\pi}_s.
\end{equation}
\subsection{It\^o formula for the non-linear filtering equation}
As for the Zakai equation, is interesting to study the It\^o formula for the composition of the process $\Pi^{s,\pi}$ with a function $u\in\mathrm{C^2_L}(\ptwo)$. The following result is stated for simplicity in the case $s=0$ and hiding the dependence on the initial condition $\pi$. Moreover it holds for a generic weak solution to the Kushner-Stratonovich equation, and not only for the ones obtained from the solutions to the Zakai equation.
\begin{proposition}\label{prop: itoks}
	Let $\{(\Omega,\mathcal{F},\{\mathcal{F}_t\},\mathbb{P}),(\Pi_t,I_t)\}$ be a weak solution to the Kushner-Stratonovich equation starting from $\pi\in\ptwo$ and let $u\in\mathrm{C^2_L}(\ptwo)$. Moreover, let us assume Hypotheses \ref{hp: generali}. Then for every $t\in[0,T]$ it holds:
	\begin{equation}\label{eqn: itoks}
		\begin{aligned}
			&u(\Pi_t) = u(\pi) + \int_0^t\Pi_s\left(\diff _\mu u(\Pi_s)\cdot f\right)\de s\\
			& \ +\int_0^t \frac{1}{2}\Pi_s\left(\tr\left\{\diff _x\diff _\mu u(\Pi_s)\sigma\sigma^\top\right\}\right)\de s+ \int_0^t\frac{1}{2}\Pi_t\left(\tr\left\{\diff _x\diff _\mu u(\Pi_s)\bar{\sigma}\bar{\sigma}^\top\right\}\right)\de s \\
			& \ + \int_0^t\frac{1}{2}\Pi_s\otimes\Pi_s\left(\lf^2u(\Pi_s)h\cdot h\right)\de s
			 + \int_0^t\frac{1}{2}\Pi_s\otimes\Pi_s\left(\tr\left\{\diff _\mu^2u(\Pi_s)\bar{\sigma}\bar{\sigma}^\top\right\}\right)\de s\\
			& \ + \int_0^t\frac{1}{2}\left[\Pi_s(h)\cdot\Pi_s(h)\right]\Pi_s\otimes\Pi_s\left(\lf^2u(\Pi_s)\right)\de s 
			 + \int_0^t\Pi_s\otimes\Pi_s\left(h\cdot\bar{\sigma}^\top\lf \diff _\mu u(\Pi_s)\right)\de s\\
			& \ - \int_0^t \Pi_s\otimes\Pi_s\left(\lf^2 u(\Pi_s)h\right)\cdot\Pi_s(h)\de s
			 \ - \int_0^t \Pi_s\otimes\Pi_s\left(\bar{\sigma}^\top\lf \diff _\mu u(\Pi_s)\right)\cdot \Pi_s(h)\de s\\
			 & \ +\int_0^t\Pi_s\left(h\lf u(\Pi_s)\right)\cdot\de I_s  +\int_0^t\Pi_s\left(\bar{\sigma}^\top \diff _\mu u(\Pi_s)\right)\cdot\de I_s - \int_0^t\Pi_s\left(\lf u (\Pi_s)\right)\Pi_s(h)\cdot\de I_s,
		\end{aligned}
	\end{equation}
	almost surely.
\end{proposition}
\begin{proof}
	The proof use the same approximation technique used in Proposition \ref{prop: itoz} proof. One has only to notice that the localization of the mass in the proof of Proposition \ref{prop: itoz} can be avoided, since $\Pi_t(\R^d) = 1$ for every $t\in[0,T]$, and the steps from one to four have to be done keeping in mind Remark \ref{rmk: lemmasprob}.
\end{proof}
\begin{remark} We can rewrite \eqref{eqn: itoks} as
	\begin{multline*}
		\de u (\Pi_t) = \Pi_t(A\lf u (\Pi_t))\de t + \Pi_t((h - \Pi_t(h) + B)\lf u(\Pi_t))\cdot\de I_t \\+ \frac{1}{2}\Pi_t\otimes\Pi_t((h - \Pi_t(h) + B)\cdot(h - \Pi_t(h) + B)\lf^2 u(\Pi_t))\de t.
	\end{multline*}
\end{remark}
\begin{remark}\label{rmk: whymyitoiscool} In the literature, in particular in the mean field games context, some It\^o formulas have been proved for the composition of $\ptwo$-valued processes and real-valued functions over $\ptwo$. A remarkable result can be found \cite[Section 4.3]{carmonadelarue2}, in which the $\ptwo$-valued process is the law of a diffusion process of the form 
\begin{equation*}
	\de X_t = f(X_t)\de s + \sigma(X_t)\de W_t + \bar{\sigma}(X_t)\de B_t,
\end{equation*} 
conditioned to $B$, where $B$ and $W$ are two independent Brownian motions. The main difference with our technique is that we have an explicit equation for the measure-valued process and we use it to deduce the It\^o formula, whilst in the approach of \cite{carmonadelarue2} the result is obtained combining the classical It\^o formula, the empirical projection of the function $u$ and the equation for the process $X$. In particular, a key tool in that approach are some formulas that relate the partial derivatives of the empirical projection $u(n^{-1}\sum_i^n\delta_{x_i})$ with the $L$-derivatives of $u$.  We can also notice that, heuristically, if we set $h$ equal to zero in the filtering problem, the Kushner-Stratonovich equation describe the law of $X$ given the filtration generated by $B$ up to a certain time. In this case, we can see that \eqref{eqn: itoks} coincide with the formula in \cite{carmonadelarue2}. Moreover, our technique also allows to deal with $\mathcal{M}_2^+(\R^d)$-valued processes, as we did in Section \ref{sec: zakai}, thanks to the fact that it is based directly on the equation for the measure-valued process and not on the fact that the measure-valued process has to be a conditional law of a finite dimensional process.
\end{remark}
\begin{remark}
	As for Corollary \ref{cor: itoztime}, an It\^o formula for $u$ depending also on time easily follows from Proposition \ref{prop: itoks}.
\end{remark}
\subsection{The backward Kolmogorov equation}
As we did for the Zakai equation, we want to discuss the existence and uniqueness of classical solutions to the backward Kolmogorov equation associated to the Kushner-Stratonovich equation \eqref{eqn: ks}. Such partial differential equation reads as
\begin{equation}\label{eqn: kolmogorovks}
	\begin{cases}
		\partial_su(\pi,s) + \mathcal{L}^{KS}u(\pi,s) = 0\quad&(\pi,s)\in\ptwo\times[0,T],\\
		u(\pi,T) = \Phi(\pi)\quad&\pi\in\ptwo,
	\end{cases}
\end{equation}
where $\Phi\in\mathrm{C^2_L}(\ptwo)$ and the operator $\mathcal{L}^{KS}\colon\mathrm{C^2_L}(\ptwo)\to\mathrm{C_b}(\ptwo)$ is defined by
\begin{equation*}
 	\mathcal{L}^{KS} u(\pi) = \pi\left(A\lf u(\pi)\right) + \frac{1}{2}\pi\otimes\pi\left((h + B - \pi(h))^\top(h + B - \pi(h))\lf^2 u(\pi)\right),
\end{equation*}
	where $A$ and $B$ are defined by \eqref{eqn: opdiff} and $h$ is Borel measurable and bounded.
\begin{definition}
	We say that $u\colon\ptwo\times[0,T]\to\R$ is a classical solution to \eqref{eqn: kolmogorovks} if it of class $\rm C^2_L(\ptwo)$ in the measure argument and $\mathrm {C}^1([0,T])$ (where in $t=0$ and $t=T$ the derivatives are understood in unilateral sense) in the time argument, if it and all its derivatives are bounded in all their arguments and if it satisfies the backward equation \eqref{eqn: kolmogorovks}. 
\end{definition}
As we did for the Kolmogorov equation associated to the Zakai equation, we want to show existence and uniqueness via a representation formula. Let $\rho^{s,\pi}$ be a solution to the Zakai equation defined over $(\Omega, \mathcal{F},\{\mathcal{F}_t\},\mathbb{Q}, \{Y\})$. Following  \eqref{eqn: ksfromz1}-\eqref{eqn: ksfromz2}, the couple $\{({\Omega},\mathcal{{F}},\{{\mathcal{F}}_t\},{\mathbb{P}^\pi}),$ $({\Pi}^{s,\pi}_t,{I^\pi_t})\}$ solves weakly the Kushner-Stratonovich equation. We can notice that the probability space $(\Omega, \mathcal{F},\mathbb{Q})$ is fixed for every $\pi\in\ptwo$, but since $I^\pi$ and $\xi^\pi$ depend on $\rho^{s,\pi}$, the probability space $(\Omega,\mathcal F,\mathbb P^\pi)$ depends on the initial point $\pi\in\ptwo$. Our claim is that 
\begin{equation*}
	u(\pi,s)\defeq \mathbb{E}^{\mathbb{P}^\pi}\left[\Phi(\Pi^{s,\pi}_T)\right] = \mathbb{E}^{\mathbb{Q}}\left[\Phi(\rho_T^{s,\pi}/\rho_T^{s,\pi}(\R^d))\xi^\pi_T\right]
\end{equation*} is the unique weak solution to \eqref{eqn: kolmogorovks}. In order to study its regularity, we rely on the relations \eqref{eqn: ksfromz2} and the regularity results obtained in Section \ref{ssec: diffmeas} for the Zakai context.
\begin{proposition}\label{prop: regks}
	Let $u(\pi,s)\defeq \mathbb{E}^{\mathbb{P}^\pi}\left[\Phi(\Pi^{s,\pi}_T)\right] = \mathbb{E}^{\mathbb{Q}}\left[\Phi(\Pi^{s,\pi}_T)\xi_T^\pi\right]$ be defined as above and let Hypotheses \ref{hp: generali} hold. Then for every $s\in [0,T]$ the mapping $u(\cdot,s)\in \mathrm{C_L^2}(\ptwo)$. Moreover, the mappings $[0,T]\ni s\mapsto \mathcal{L}^{KS}u(\pi,s)$ and $[s,T]\times[0,T]\ni(\tau,\sigma)\mapsto\mathcal{L}^{KS}u(\Pi_\tau^{s,\pi},\sigma)\in L^2(\Omega,\mathbb {P}^\pi)$ are continuous.
\end{proposition}
\begin{proof}
	First, since $\de {\rho^{s,\pi}}(\R^d) = \rho^{s,\pi}(\R^d)\scalprod{\Pi^{s,\pi}_t}{h}\cdot\de Y_t$, the process $1/{\rho^{s,\pi}}(\R^d)$ has uniformly in time bounded $\mathbb Q$ moments of any order $p\in[1,+\infty)$. Then, let us compute the linear functional derivative of $u$ for a fixed $s\in[0,T)$:
	\begin{multline*}
		\delta_\pi u (\pi,x,s) =\mathbb{E}^{\mathbb Q}\left[\xi_T^\pi \delta_\pi\left( \Phi\left(\frac{\rho_T^{s,\cdot}}{\rho_T^{s,\cdot}(\R^d)}\right)\right)(\pi,x) + \Phi\left(\frac{\rho_T^{s,\cdot}}{\rho_T^{s,\cdot}(\R^d)}\right)\delta_\pi \xi_T^\cdot(\pi,x)\right]\\
		=\mathbb{E}^{\mathbb Q}\left[ \frac{\xi^\pi_T}{\rho_T^{s,\pi}(\R^d)}\scalprod{Z^s_T(x)}{\lf \Phi\left(\frac{\rho_T^{s,\pi}}{\rho_T^{s,\pi}(\R^d)}\right)}\right]\\
		-\mathbb{E}^{\mathbb Q}\left[  \frac{\xi^\pi_T}{\rho_T^{s,\pi}(\R^d)}\scalprod{\frac{\rho_T^{s,\pi}}{\rho_T^{s,\pi}(\R^d)}}{\lf \Phi\left(\frac{\rho_T^{s,\pi}}{\rho_T^{s,\pi}(\R^d)}\right)}\scalprod{Z^s_T(x)}{\mathbf 1}\right] \\
		 + \mathbb{E}^{\mathbb Q}\left[\Phi\left(\frac{\rho_T^{s,\pi}}{\rho_T^{s,\pi}(\R^d)}\right)\xi_T^\pi\int_s^T\frac{1}{\rho_\tau^{s,\pi}(\R^d)}\left(\scalprod{Z^s_\tau(x)}{h} - \frac{\scalprod{\rho_\tau^{s,\pi}}{h}}{\rho_\tau^{s,\pi}(\R^d)}\scalprod{Z^s_\tau(x)}{\mathbf 1}\right)\cdot\de Y_\tau\right]\\
		  - \mathbb{E}^{\mathbb Q}\left[\Phi\left(\frac{\rho_T^{s,\cdot}}{\rho_T^{s,\cdot}(\R^d)}\right)\xi_T^\pi \int_s^T\frac{\scalprod{\rho_\tau^{s,\pi}}{h}}{\rho_\tau^{s,\pi}(\R^d)^2}\cdot\left(\scalprod{Z^s_\tau(x)}{h} - \frac{\scalprod{\rho_\tau^{s,\pi}}{h}}{\rho_\tau^{s,\pi}(\R^d)}\scalprod{Z^s_\tau(x)}{\mathbf 1}\right)\de \tau\right],
	\end{multline*}
	where we computed $\delta_\pi \xi_T^\cdot(\pi,x)$ thanks to stochastic and deterministic Fubini's theorem. Continuity and boundedness are guaranteed by the regularity and boundedness of the processes involved under the $\mathbb{Q}$-expectation. In the same way one can show the second-order differentiability in linear functional sense of $u$. Regarding the differentiability in space, again we can bring the derivative in space inside the expectation and exploit the regularity results in Proposition \ref{prop: summarize}. To conclude, the continuity of the mappings $s\mapsto \mathcal{L}^{KS}u(\pi,s)$ and $(\tau,\sigma)\mapsto\mathcal{L}^{KS}u(\Pi_\tau^{s,\pi},\sigma)\in L^2(\Omega,\mathbb {P}^\pi)$ follows as in Lemma \ref{lemma: contofl}.
\end{proof}

Finally, we can state the existence and uniqueness result for the Kolmogorov equation associated to the Kushner-Stratonovich equation \eqref{eqn: ks}:
\begin{theorem}\label{thm: exuniqkskolm}
	Let Hypotheses \ref{hp: generali} holds and let $\{({\Omega},\mathcal{{F}},\{{\mathcal{F}}_t\},{\mathbb{P}^\pi}),({\Pi}^{s,\pi}_t,{I^\pi_t})\}$ be the weak solution to the Kushner-Stratonovich equation obtained trough \eqref{eqn: ksfromz1}-\eqref{eqn: ksfromz2}. There exists a unique classical solution to the backward Kolmogorov equation \eqref{eqn: kolmogorovks} starting at $\Phi\in\mathrm{C^2_L}(\R^d)$, given by 
	\begin{equation*}
		u(\pi,s) = \mathbb{E}^{\mathbb P^\pi}\left[\Phi(\Pi_T^{s,\pi})\right],\quad (\pi,s)\in \ptwo\times [0,T].
	\end{equation*}
\end{theorem}
\begin{proof}
	The proof follows exactly the one we did in Section \ref{sec: exuniqzakai} for the backward Kolmogorov equation associated to the Zakai equation, exploiting the regularity results in Proposition \ref{prop: regks} and the Markov property for the existence, and the It\^o formula in Proposition \ref{prop: itoks} for the uniqueness.
\end{proof}
\begin{appendix}
\section{Proof of Lemma \ref{lemma: solinm2p2}}\label{app: a} 
For simplicity, we provide a sketch of the proof for the Zakai equation in the one dimensional case ($d=1$). The general case with $d>1$ and the case of Kushner-Stratonovich equation are immediate extensions. In order to keep the notation lighter, we remove the tildes in the notation for the weak solutions and we will denotes the expectation $\mathrm{E}^{\tilde{\mathbb{Q}}}$ with respect to $\tilde{\mathbb{Q}}$ just with $\mathbb{E}$.\\
	
	First, we show that if $\mu\in\mathcal{M}_1^+(\R)$, then $\rho_t\in\mathcal{M}_1^+(\R)$ for every $t\in[0,T]$, almost surely. Let us consider a smooth function $\psi\colon\R\to\R$ which is greater than the mapping $x\mapsto|x|$ in a neighbourhood of $0$ and equal to $x\mapsto|x|$ outside that neighbourhood. In particular, $\psi$ has bounded first and second-order derivatives. If we show that $\scalprod{\rho_t}{\psi}<\infty$ for every $t\in[0,T]$, almost surely, then this first claim is proved.\\
	
	Let us consider an increasing family of smooth cut-off functions $\left\{\phi^N\right\}_{N\geq 1}$ which are equal to one in $[-N, N]$ and equal to zero outside $[-N-1,N+1]$. These functions can be chosen to be bounded together with their first and second-order derivatives by a constant $C>0$ independent of $N$. We set $\psi^N(x) = \psi(x)\phi^N(x)$ for every $x\in\R$, thus it holds
	\begin{align*}
		\diff _x \psi^N(x) & =\psi(x) \diff _x \phi^N(x) + \diff _x\psi(x) \phi^N(x), \\
		\diff ^2_x \psi^N(x) &= \psi(x) \diff ^2_x \phi^N(x) + 2\diff _x\psi(x) \diff _x\phi^N(x) +  \diff ^2_x\psi(x) \phi^N(x),
	\end{align*}
	and moreover $\{\psi^N\}_{N\geq1}, \{\diff _x \psi^N\}_{N\geq1}$ and  $\{\diff ^2_x \psi^N\}_{N\geq1}$ converge pointwise to $\psi, \diff _x\psi$ and $\diff _x^2\psi$ respectively, where the first convergence takes place monotonically. We also notice that for every $N\geq 1$, $\psi^N\in\mathrm{C^2_b}(\R)$.\\

	Let us fix $N\geq 1$. Since $\rho$ is a weak solution to the Zakai equation, it holds for every $t\in[0,T]$ 
	\begin{equation*}
		\scalprod{\rho_t}{\psi^N} = \scalprod{\mu}{\psi^N} + \int_0^t \scalprod{\rho_s}{A\psi^N}\de s + \int_0^t \scalprod{\rho_s}{(h + B)\psi^N}\cdot \de Y_s.
	\end{equation*}
	By taking the square, the expectation and then by It\^o isometry, we get
	\begin{equation*}
		\E{\scalprod{\rho_t}{\psi^N}^2}\leq 3\scalprod{\mu}{\psi^N}^2 + 3T\int_0^t \E{\scalprod{\rho_s}{A\psi^N}^2}\de s + 3\int_0^t \E{\scalprod{\rho_s}{(h + B)\psi^N}^2}\de s.
	\end{equation*}
	Now, if we write explicitly the operators $A, B$ and we use the boundedness of $b,\sigma,\bar\sigma, h$ jointly with the boundedness of $\diff _x\psi,\diff ^2_x\psi, \phi^N,\diff _x\phi^N,\diff _x^2\phi^N$ (recalling that the bound for $\phi^N,\diff _x\phi^N,\diff _x^2\phi^N$ does not depend on $N$), we obtain the inequality
	\begin{equation}\label{eqn: stimaphin}
		\E{\scalprod{\rho_t}{\psi^N}^2}\leq M_1\left(\scalprod{\mu}{\psi^N} ^2+ \int_0^t \E{\scalprod{\rho_s}{\psi^N}^2} + \E{\rho_s(\R)^2}\de s\right),
	\end{equation}
	where $M_1 = M_1(b,\sigma,\bar\sigma, h, T, C, \psi)$ is a positive constant independent of $t$ and $N$. Thanks to the monotone convergence theorem, we can pass to the limit as $N\to + \infty$ in \eqref{eqn: stimaphin}  and get
	\begin{equation*}
		\E{\scalprod{\rho_t}{\psi}^2}\leq M_1\left(\scalprod{\mu}{\psi}^2 + \int_0^t \E{\scalprod{\rho_s}{\psi}^2} + \E{\rho_s(\R)^2}\de s\right).
	\end{equation*}
	Thus, in view of Remark \ref{rmk: bddmass} and by Gronwall's lemma, there exists a positive constant $M_2 = M_2(b,\sigma,\bar\sigma, h, T, C, \psi,\mu)$ such that for every $t\in[0,T]$ it holds that
	\begin{equation}\label{eqn: boundmomentoprimotfixed}
		\E{\scalprod{\rho_t}{\psi}^2}\leq M_2\scalprod{\mu}{\psi}^2.
	\end{equation}
	Since the bound \eqref{eqn: boundmomentoprimotfixed} does not depend on $t$, we can proceed similarly to the previous steps and by Burkholder inequality and monotone convergence we can also deduce that there exists a positive constant $M_3 = M_3(b,\sigma,\bar\sigma, h, T, C, \psi,\mu)$ such that
	\begin{equation}\label{eqn: bddfirstmoment}
		\E{\sup_{0\leq t\leq T}\scalprod{\rho_t}{\psi}^2}\leq M_3\scalprod{\mu}{\psi}^2,
	\end{equation}
	thus if $\mu\in\mathcal{M}_1^+(\R)$ then $\sup_{0\leq t\leq T}\scalprod{\rho_t}{\psi}<+\infty$ almost surely and so $\rho_t\in\mathcal{M}_1^+(\R)$ for every $t\in[0,T]$, almost surely.\\
	
	To conclude, we need to prove that if $\mu\in\mathcal{M}_2^+(\R)$ then $\rho_t\in\mathcal{M}_2^+(\R)$ for every $t\in[0,T]$, almost surely. To this aim, we can proceed analogously to the above case in which $\mu\in\mathcal{M}_1^+(\R)$, choosing $\psi(x) = x^2$ and noticing that its first and second-order derivatives are linear and constant respectively. Then, we can still use the approximation technique, combined with Remark \ref{rmk: bddmass} and \eqref{eqn: bddfirstmoment} and so the lemma is proved.

\end{appendix}
\subsection*{Acknowledgements} The author thanks professor Marco Fuhrman for bringing this problem to his attention and for the several useful discussions, and professor Boualem Djehiche for the helpful discussions during his stay in Milan. The author would like to thank the  anonymous referees for their very helpful comments and suggestions from which the manuscript has benefited.

\printbibliography
\end{document}